\documentclass[11pt]{amsart}
\usepackage{amsmath,amsthm,amscd,amssymb, color}
\usepackage{latexsym}
\usepackage[colorlinks, citecolor = blue]{hyperref}
\usepackage[margin = 1.333in]{geometry}
\usepackage{cases}

\newcommand{\Q}{\mathbf Q}
\newcommand{\C}{\mathbf C}

\newcommand{\F}{\mathbb F}

\newcommand{\p}{\mathfrak p}
\renewcommand{\phi}{\varphi}

\newcommand{\eps}{\varepsilon}
\renewcommand{\O}{\mathcal O}
\newcommand{\Ucal}{\mathcal U}
\newcommand{\A}{\mathbf A}
\newcommand{\bs}{\backslash}

\newcommand{\1}{\mathbf 1}
\newcommand{\bmx}{\left( \begin{matrix}}
\newcommand{\emx}{\end{matrix} \right)}

\renewcommand{\~}{\widetilde}

\newcommand{\Xef}{{X(E\!:\!F)}}
\newcommand{\Xefpi}{{X(E\!:\!F\!:\!\pi')}}
\renewcommand{\ell}{\mathrm{ell}}
\newcommand{\JL}{\mathrm{JL}}
\newcommand{\LJ}{\mathrm{LJ}}

\DeclareMathOperator{\semi}{ss}
\DeclareMathOperator{\reg}{reg}\DeclareMathOperator{\main}{main}
\DeclareMathOperator{\GL}{GL} \DeclareMathOperator{\GSp}{GSp} 
\DeclareMathOperator{\Gal}{Gal}  \DeclareMathOperator{\Hom}{Hom}
  
  \DeclareMathOperator{\vol}{vol}
\DeclareMathOperator{\SL}{SL}  
  
 \DeclareMathOperator{\PGL}{PGL} \DeclareMathOperator{\SO}{SO}

  \DeclareMathOperator{\twist}{tw}

\newtheorem{lem}{Lemma}
\numberwithin{lem}{section}
\newtheorem{prop}[lem]{Proposition}
\newtheorem{thm}[lem]{Theorem}
\newtheorem{defn}[lem]{Definition}
\newtheorem{cor}[lem]{Corollary}
\newtheorem{conj}[lem]{Conjecture}
\newtheorem{rem}{Remark}
\numberwithin{rem}{section}
\theoremstyle{definition}

\numberwithin{equation}{section}

\title[Periods and nonvanishing for $\GL(2n)$]{Periods and nonvanishing of central $L$-values for $\GL(2n)$}

\author[Brooke, Kimball, Dave]{Brooke Feigon, Kimball Martin, David Whitehouse}
 \address{Brooke Feigon: Department of Mathematics,
The City College of New York,
CUNY,
NAC 8/133,
New York, NY 10031} \email{bfeigon@ccny.cuny.edu }
\address{Kimball Martin: University of Oklahoma, 
Department of Mathematics, 
Norman, OK 73019}
\email{kmartin@math.ou.edu}
\address{David Whitehouse}
\begin{document}

\date{\today}



\maketitle

\begin{abstract} Let $\pi$ be a cuspidal automorphic representation of $\PGL(2n)$ over a number field $F$,
and $\eta$ the quadratic id\`ele class character attached to a quadratic extension $E/F$.
Guo and Jacquet conjectured a relation between the nonvanishing of $L(1/2,\pi)L(1/2, \pi \otimes \eta)$ for $\pi$ of symplectic type and the nonvanishing of certain $\GL(n,E)$ periods. 
When $n=1$, this specializes to a well-known result of Waldspurger.  
We prove this conjecture, and related global results, under some local hypotheses using a simple relative trace formula.  

We then apply these global results to obtain local results on distinguished supercuspidal 
representations, which partially establish a conjecture of Prasad and Takloo-Bighash.
\end{abstract}


\section{Introduction}

One topic of recent interest is the study of how
periods behave along functorial transfers, for instance among inner forms \`a la Gross--Prasad conjectures.  The relative trace formula is an analytic tool developed for such problems.
Here we apply a simple relative trace formula to show that certain periods behave
as conjectured by Guo and Jacquet with respect to the Jacquet--Langlands lift under some
local hypotheses.  While these local hypotheses are somewhat restrictive, 
they are merely technical hypotheses used to simplify the trace formula and
we can verify them sufficiently often to
provide strong evidence for both the Guo--Jacquet conjecture and the feasibility
of a trace formula approach.

Then we apply our global results to obtain local results on distinguished supercuspidal representations.  These local results establish part of \cite[Conjecture 1]{PTB} and 
\cite[Conjecture 3]{FM-conj}, which concern
local root number criteria for the existence of certain local linear forms and global periods.

\subsection{Background}
Let $E/F$ be a quadratic extension of number fields, and $\eta$ the associated quadratic
id\`ele class character.  Let $\A$ be the ad\`eles of $F$ and $\A_E$ the ad\`eles of $E$.
Let $\Xef$ denote the set of quaternion
algebras $D/F$ in which  $E$ embeds.  Note the matrix algebra $M_2$ always lies in $\Xef$.  
For $D \in \Xef$, let $G_D = \GL_n(D)$.  When $D$
is fixed, we also write $G=G_D$.   Put $G'=\GL_{2n}$.  For each $G_D$, at almost all places
$v$ of $F$, we have $G_{D_v} \cong G'_v$. 
The Jacquet--Langlands correspondence in \cite{badulescu:renard} 
 associates to each discrete series 
representation $\pi_D$ of $G_D(\A)$ a discrete series representation $\pi' = \JL(\pi_D)$ of $G'(\A)$ such that $\pi_{D_v} \cong \pi'_v$ for almost all $v$.  
Strong multiplicity one for these groups means this
near local equivalence condition specifies a unique $\pi'$ for each $\pi_D$, and vice versa when
the inverse Jacquet--Langlands lift exists.

All of our representations are taken to be unitary with trivial central character, and we will also
assume both $\pi_D$ and $\pi'$ are cuspidal.

We now describe the periods of interest.  Let $H=H_D$ be the subgroup $\GL_n(E)$ of $G_D$ 
and $H'$ be the subgroup
$\GL_n \times \GL_n$ of $G'$ (all embeddings of these subgroups are conjugate, but we
fix embeddings in Section~\ref{notation-sec}).  
Denote by $dh$ and $dh'$ Haar measures on $H$ and $H'$. 
We say $\pi_D$ is $H$-distinguished if the linear form on $\pi_D$ given by
\[ \mathcal P_D(\phi) = \int_{H(F)Z(\A) \bs H(\A)} \phi(h)\, dh \]
is not identically zero.  (Throughout $Z$ denotes the center of the relevant group.) Now consider the linear forms on $\pi'$ given by
\[ \mathcal P'(\phi) = \int_{H'(F)Z(\A) \bs H'(\A)} \phi(h') \, dh', \quad
 \mathcal P_\eta'(\phi) = \int_{H'(F)Z(\A) \bs H'(\A)} \phi(h') \eta(\det(h')) \, dh'. \]
We say that $\pi'$ is 
$H'$- (resp.\ $(H',\eta)$-) distinguished if the linear form $\mathcal P'$ (resp.\  $\mathcal P'_\eta$)
is not identically 0.

These periods are intimately connected with central $L$-values.  Specifically, we have the 
following consequence of a result of Friedberg and Jacquet.
\begin{thm} [\cite{friedberg:1993}] Suppose $\pi'$ is cuspidal.  Then $\pi'$ is both $H'$- and 
$(H',\eta)$-distinguished if and 
only if (a) $\pi'$ is of symplectic type, i.e., $L(s, \pi', \Lambda^2)$ has a pole at $s=1$, and (b)
\[ L(1/2, \pi'_E) = L(1/2, \pi') L(1/2, \pi' \otimes \eta) \ne 0.  \]
\end{thm}
Here $\pi'_E$ denotes  the base change of $\pi'$ to $G'(\A_E)$.
We remark that being of symplectic type on $\GL_{2n}$ is equivalent to being in the image of the functorial
lift from $\SO_{2n+1}$ (see Arthur's book \cite{arthur:book}).

Given $\pi'$, let $\Xefpi$ denote the set of $D \in \Xef$ such that the inverse Jacquet--Langlands
transfer $\pi_D=\LJ_D(\pi)$ of $\pi'$ to $G_D$ exists.  In particular, $M_2 \in \Xefpi$.  When $n=1$, we have
the following well-known theorem of Waldspurger.  In this case, all $\pi'$ are of symplectic type.

\begin{thm}[\cite{waldspurger:1985}] \label{wald-thm} 
Suppose $n=1$.  

\begin{enumerate}
\item Fix $D \in \Xef$ and a cuspidal representation $\pi$ of $G_D(\A)=D^\times(\A)$.  If $\pi$ is
$H$-distinguished, then $\pi' = \JL(\pi)$ is $H'$- and $(H',\eta)$-distinguished.

\item Let $\pi'$ be a cuspidal representation of $G'(\A)=\GL_2(\A)$.  If $\pi'$ is 
$H'$- and $(H',\eta)$-distinguished, then there exists a (unique) $D \in \Xefpi$ such that
$\pi_D = \LJ_D(\pi')$ is $H$-distinguished.
\end{enumerate}
\end{thm}

This result (and a twisted version) was originally proved
by Waldspurger using
the theta correspondence, and then by Jacquet \cite{jacquet:1986}
using the relative trace formula.  In fact, Waldspurger obtained an exact formula  
relating $L(1/2,\pi'_E)$ to $|\mathcal P_D(\phi)|^2$ for any $\phi \in \pi_D$.  This formula,
and its twisted version, has been the subject of much study and has many applications.

One might hope to generalize Waldspurger's result to higher rank groups.  
One such generalization is the orthogonal (Gan--)Gross--Prasad conjectures 
(\cite{gross:prasad1}, \cite{gross:prasad2}, \cite{GGP}) viewing 
Theorem \ref{wald-thm} as a statement about $\SO_3 \times \SO_2$ $L$-values.  
Alternatively, remaining in
the general linear framework, we have the following.

\begin{conj} [Guo--Jacquet \cite{guo:1996a}] 
\begin{enumerate}
\item Fix $D \in \Xef$ and let $\pi$ be a cuspidal representation of $G(\A) = G_D(\A)$.  
If $\pi$ is $H$-distinguished, then $\pi'=\JL(\pi)$ is $H'$- and $(H',\eta)$-distinguished, i.e.,
$\pi'$ is of symplectic type and $L(1/2,\pi'_E) \ne 0$.

\item Suppose $n$ is odd, and $\pi'$ is a cuspidal representation of $G'(\A)$ such 
that $\pi'$ is $H'$- and $(H',\eta)$-distinguished.  Then there exists $D \in \Xefpi$ such that $\pi_D = \LJ_D(\pi')$
is $H$-distinguished.
\end{enumerate}
 \label{guo-conj}
\end{conj}

Guo \cite{guo:1996a} established the fundamental lemma for the
unit element of the Hecke algebra for a relative trace formula
to attack this conjecture.  

We remark that the condition of $n$ odd for (2) of the
conjecture arises in a difference between the geometric decompositions of this trace formula
in the $n$ odd and $n$ even case.  Spectrally, the difference is related to 
the sign in the character identity $\chi_{\pi_v} = (-1)^n \chi_{\pi_v'}$ of the local Jacquet--Langlands correspondence.  

The correspondence of geometric orbits
in the trace formula suggests the $D$ in (2) may be unique when $n$ is odd.
This suggests a
local dichotomy principle as in the Gross--Prasad situation, i.e., 
when $G_v$ is nonsplit, at most one of $\pi_v$ and 
$\pi'_v$  is locally $H(F_v)$- or $H'(F_v)$-distinguished (has a nonzero $H(F_v)$- or $H'(F_v)$-invariant linear form).

Local dichotomy when $n$ is odd would also be consistent with spectral
expectations.  Let $K/k$ be a quadratic extension of local fields  and $D(k)$ the quaternion division algebra over $k$.

\begin{conj} \cite[Conjecture 1]{PTB} \label{PTB-conj}  Let $\tau$ and $\tau'$ be irreducible admissible representations of $\GL_n(D(k))$ and $\GL_{2n}(k)$ which correspond via
Jacquet--Langlands.  If $\tau$ (resp.\ $\tau'$) is $\GL_n(K)$-distinguished, 
then $\tau$ (resp.\ $\tau'$) is symplectic
and $\epsilon(1/2, \tau_{K}) = (-1)^n$ (resp.\ $\epsilon(1/2, \tau'_{K}) = 1$).  Moreover, these conditions are sufficient
for distinction if $\tau$ (resp.\ $\tau'$) is  discrete series.
\end{conj}

In fact, \cite{PTB} gives a more general conjecture and proves
 the $n=2$ case with the theta correspondence.  Symplectic here means the local Langlands parameter has 
symplectic image in $\GL_{2n}(\C)$.

This conjecture implies that (i) when $n$ is odd, at most one of $\tau$ and $\tau'$ are locally $\GL_n(K)$-distinguished; and (ii) when $n$ is even and $\tau$ is discrete series, then $\tau$ is locally $\GL_n(K)$-distinguished if and only if $\tau'$ also is.  Hence, at
least for discrete series representations, one should have local dichotomy precisely when $n$ is odd.

On the other hand, one might ask for an analogue of (2) when $n$ is even.
Here it appears that extra conditions are needed to get $D$ with $\pi$ distinguished (cf.\
\cite{PTB}, \cite{FM-conj}).  Moreover, when such a $D$ exists, it need not be 
unique---this is suggested by (ii) and we will prove this below.
We hope to discuss an analogue of (2) for $n$ even in future work.

\subsection{Main results}
Building on Guo's work, we establish a simple relative trace formula to prove our
main global result.

\begin{thm} \label{intro-main-thm}
Suppose $E/F$ is split at all archimedean places.
Assume $\pi$ is supercuspidal at some place
which splits in $E$ and $H$-elliptic at another place.  Then 
Conjecture \ref{guo-conj}(1) holds, 
i.e., if $\pi=\pi_D$ is $H$-distinguished, then $L(1/2, \pi'_E) \ne 0$
and $\pi'$ is symplectic.

\end{thm} 

This is Theorem \ref{main-thm1} below.  
The condition of being $H$-elliptic at some place means that the associated
local Bessel distribution is nonzero on the ``$H$-elliptic'' set (see Section 
\ref{ell-supp-sec}).  This condition holds for many representations and will be discussed 
momentarily.

One might also hope to show the converse, Conjecture \ref{guo-conj}(2), when
$n$ is odd.  This is more difficult due to the nature of the geometric correspondence in
the relative trace formula.
Still, in Proposition \ref{main-thm2} we prove a converse
result under some hypotheses, for $n$ even or odd, though the hypotheses are weaker
when $n$ is odd.

Our final global result, Theorem \ref{main-thm3},
gives sufficient conditions for an $H$-period of $\pi_{D_1}$
to transfer to an $H$-period of $\pi_{D_2}$, for two $D_1, D_2 \in \Xefpi$, when $n$ is even.
This tells us that when an analogue of Conjecture \ref{guo-conj}(2) for $n$ even holds,
the $D$ should not be unique.

\medskip
Since being globally distinguished implies
being locally distinguished at each place, an appropriate global embedding result for locally distinguished 
representations allows us to  conclude the following local results.

Let $K/k$ be a quadratic extension of $\mathfrak p$-adic fields, $\eta_{K/k}$
the associated quadratic character, and $D(k)$ the quaternion division
algebra over $k$.  

Put $H(k) = \GL_n(K)$ and $H'(k) = \GL_n(k) \times \GL_n(k)$.  

\begin{thm} \label{intro-local-thm}
Let $\tau$ be a supercuspidal representation of $\GL_n(D(k))$ and $\tau'$ its
Jacquet--Langlands transfer to $\GL_{2n}(k)$.

(a) If $\tau$ is $H(k)$-distinguished, then $\tau'$ is both $H'(k)$- and
$(H'(k),\eta_{K/k})$-distinguished.

(b) Suppose $n$ is even and $\tau'$ is also supercuspidal.  If one of
$\tau$ and $\tau'$ is both $H(k)$-distinguished and $H(k)$-elliptic, then the other also is.
\end{thm}

This is contained in Theorems \ref{local-thm1} and \ref{local-thm2} below, and establishes part of consequence (ii) of Conjecture \ref{PTB-conj} under an additional elliptic assumption.

Using a similar idea to \cite{prasad:duke}, we also obtain one direction of Conjecture \ref{PTB-conj} for  supercuspidal representations of $\GL_{2n}$:

\begin{thm} \label{intro-local-thm2} 
Let $\tau'$ be a supercuspidal representation of $\GL_{2n}(k)$.
If $\tau'$ is $\GL_n(K)$-distinguished, then
$\tau'$ is symplectic and $\epsilon(1/2, \tau'_K) =1$.
\end{thm}

Let us now discuss the global result in more detail.

We expect that our approach to Theorem \ref{intro-main-thm}
should lead to a formula for the $L$-value $L(1/2, \pi'_E)$
in terms of the square periods $|\mathcal P_D(\phi)|^2$, as in Waldspurger's case.  This was
carried out with a relative trace formula for $n=1$ by Jacquet--Chen \cite{jacquet:2001} and
the latter two authors \cite{me:periods}.
In higher rank, Wei Zhang \cite{wei2} also used a simple relative trace formula to obtain an $L$-value formula in the setting of the unitary Gan--Gross--Prasad conjectures under some local hypotheses.   

When $n=2$, this theorem can be thought of as a relation between the nonvanishing  of certain periods and the
nonvanishing of central spinor $L$-values for $\GSp(4)$.
One direction of the $\SO(5) \times \SO(2)$ case of Gross--Prasad says the nonvanishing of these
$L$-values should also be detected by Bessel periods.  This is now known, e.g.,
\cite{furusawa:morimoto}.  See \cite{FM-conj} for a
discussion of the comparison between Bessel periods and $\GL(n,E)$ periods.

The idea of proof is similar to some other recent works such as \cite{jacquet:martin}, \cite{furusawa:martin} and \cite{wei1}.  Before outlining the proof, let us highlight a couple of differences
from these other works.  First, unlike \cite{jacquet:martin} and \cite{furusawa:martin},
this is valid in higher rank and 
we use Ramakrishnan's mild Chebotarev result for $\GL(n)$ \cite{ramakrishnan:mult1} 
to avoid the need
of the full fundamental lemma for the Hecke algebra.  These are also features of
\cite{wei1}, which was completed while we were finishing this project.  Second,
we need to show we can choose a test function which has
$H$-elliptic support.  There is no need for this in the cases treated by
\cite{furusawa:martin} and \cite{wei1}, as the orbital integrals converge for a
dense set of elements.  In \cite{jacquet:martin}, this type of result was established
at an archimedean place for $\GL_2(D)$ via explicit Lie algebra calculations.
Here we impose this condition as the $H$-elliptic local hypothesis, but we
expect this to hold for most $H$-distinguished representations.  As evidence, we give
a local proof of the existence of local $H$-elliptic supercuspidals when $E_v=\Q_2\oplus\Q_2$ (Proposition \ref{ell-supp-prop}), and use this to give a
global proof of the existence of  local $H$-elliptic  
representations for any local quadratic extension (Proposition \ref{ell-supp-cor}).  These types of results and proof appear
somewhat novel.  Some results in a similar vein have recently been obtained in other cases, 
such as \cite[Proposition 9.6]{jacquet:martin} and \cite{czhang:orthog}, 
but these have very different proofs.  
Related results in the group case have been well known for some time, 
e.g.\ \cite[Proposition 2.7]{rogawski:1983}.

\subsection{Outline of method}
Now we outline the proof of Theorem \ref{intro-main-thm}.
Fix $D \in \Xef$.
The trace formula on $G_D$ is an expression of the form
\begin{equation}
 I_D(f_D) = \sum_{\gamma_D} I_{D,\gamma_D}(f_D) = \sum_{\sigma_D} I_{D,\sigma_D}(f_D),
\end{equation}
where $f_D$ is a certain test function on $G_D(\A)$, $I_{D,\gamma_D}$ are certain orbital integrals
indexed by double cosets $H(F) \bs G_D(F)/ H(F)$,
and $I_{D,\sigma_D}$ are certain spectral distributions, which for $\sigma_D = \pi_D$
 involves $\mathcal P_D$ and $\overline{\mathcal P_D}$.  A key point is that $I_{D,\pi_D} \not
 \equiv 0$ if and only if $\pi_D$ is $H$-distinguished.
 Similarly, we have a trace formula on $G'$ of the form
\begin{equation}
 I'(f') = \sum_{\gamma'} I'_{\gamma'}(f') = \sum_{\sigma'} I'_{\sigma'}(f').
\end{equation}
Here $I'_{\pi'} \not \equiv 0$ if and only if $\pi'$ is $H'$- and $(H',\eta)$-distinguished.

These trace formulas will not be convergent in general, but if we pick $f_D = \prod f_{D,v}$ and 
$f' = \prod f'_v$ so that at one place they are supported on ``regular elliptic'' double cosets and at
another place they are matrix coefficients of a supercuspidals, then both sides converge absolutely.  

One defines a correspondence among regular elliptic double cosets
$\gamma_D$ and $\gamma'$, and thus a
notion of matching functions $f_D$ and $f'$ in the sense that $I_{D,\gamma_D}(f_D) =
I'_{\gamma'}(f')$ for matching $\gamma_D$ and $\gamma'$.  Here each $\gamma_D$
corresponds to a unique $\gamma'$, and no two double cosets $\gamma_D$ correspond
to the same $\gamma'$ (for a fixed $D$ when $n$ is even, or among all $D$'s when $n$ is odd).   
However not all $\gamma'$'s correspond to a 
$\gamma_D$ when $n$ is even.  If $I'_{\gamma'}(f') = 0$ for such ``bad'' $\gamma'$,
then for $f'$ matching an $f_D$ we can write
\begin{equation} \label{rtf-comp-neven}
 I_D(f_D) = I'(f')
\end{equation}
for $n$ even.  When $n$ is odd, there are no such bad (elliptic) $\gamma'$, but a single $f'$ should
match with a family $(f_{D'})$ as $D'$ ranges over $\Xef$, and for such test functions we will
have
\begin{equation} \label{rtf-comp-nodd}
 \sum_{D' \in X(E:F)} I_{D'}(f_{D'}) =  I'(f'). 
\end{equation}
This should give the reader some sense of the differences between the $n$ even and
$n$ odd case for Conjecture \ref{guo-conj}(2). 
Since we are just proving the other direction, we may take $f'$ so that
$I_{D'}(f_{D'}) = 0$ for all $D' \ne D$.  Thus we may work with \eqref{rtf-comp-neven}
in both the $n$ odd and $n$ even cases.

Now the standard thing to do is use the fundamental lemma for the Hecke algebra 
and linear independence of characters to deduce that 
$I_{D,\pi_D}(f) = I_{\pi'}(f')$. 
In our setup we do not yet know the
fundamental lemma for the Hecke algebra, but it holds trivially at places where $E/F$ splits.  Instead, we use a result of Ramakrishnan \cite{ramakrishnan:mult1} which 
tells us that if two representations $\sigma_1'$ and $\sigma_2'$ of $\GL_{2n}$ are locally
equivalent at almost all places where $E/F$ splits, then $\sigma_2'$ is isomorphic to either
$\sigma_1'$ or $\sigma_1' \otimes \eta$.  This yields
\begin{equation} \label{Ipi-comp}
 I_{D, \pi_D}(f_D) + I_{D,\pi_D \otimes \eta}(f_D)= I_{\pi'}(f') + 
I_{\pi' \otimes \eta}(f').
\end{equation}

The main point is to show that we have sufficiently many pairs of matching functions
$(f_D)_D$ and $f'$, which reduces to a question of proving the existence of local matching
functions, as our geometric orbital integrals factor into local ones.  
Local matching comes for free when $E_v/F_v$ is split, so we may just consider
local matching at inert places.  

When $f_{D,v}$ and $f'_v$
are the unit elements of the Hecke algebra, this is the fundamental lemma proved by Guo
\cite{guo:1996a} (at almost all places).  
In Section \ref{orbint-sec}, we prove the existence of matching functions 
supported on both certain dense and elliptic subsets of $G_D$ and $G'$.  Then in Section 
\ref{bess-sec}, we use
an extension of another result of Guo \cite{guo:spherical} on local integrability of
 local Bessel distributions of $G'$ to say that at odd places
it is enough to just consider functions $f'$ supported on dense subsets of $G'$ at odd
nonarchimedean places.  Here is where the assumption about being split at
archimedean places arises.  (We  use \cite{chong-zhang} to treat even places.)

The $H$-elliptic condition allows us to choose test functions that guarantee the convergence of
the geometric sides, whereas the supercuspidal condition allows us to choose functions
that give convergence of the spectral sides.
 
 This matching is now enough to get Theorem \ref{intro-main-thm}, by showing that
$I_{\pi_D \otimes \eta}(f_D)$ cannot cancel out $I_{\pi_D}(f_D)$ for all such $f_D$, i.e., the left
hand side of \eqref{Ipi-comp} is nonzero for some $f_D$ if $\pi_D$ is $H$-distinguished,
i.e., if $I_{\pi_D} \not \equiv 0$.

\subsection{Further remarks}
After we completed an earlier draft, Chong Zhang
\cite{chong-zhang} used an idea of Wei Zhang \cite{wei1} to get a smooth transfer result of the form:
each $f_{D,v}$ has a matching $f'_v$ for a nonarchimedean place $v$.  
While we use this at even places, one could also use this at all nonarchimedean places
instead of our restricted smooth matching results in Section \ref{orbint-sec}.
However, we hope that our original approach may still be of interest as: (i) proving
restricted smooth matching is much simpler than a full smooth matching result; (ii) 
it involves reducing relative orbital integrals to orbital integrals in \cite{arthur:1989} on lower
rank groups, which may indicate an interesting connection between the trace formula
in \cite{arthur:1989} and the relative trace formula here; and (iii) this approach may be useful
in other situations where a complete smooth matching result is not known---e.g., the
archimedean situation here or cases involving other groups.

\subsection*{Acknowledgements}
We are grateful to Masaaki Furusawa, Herv\'e Jacquet, Fiona Murnaghan, Dipendra Prasad, Dinakar Ramakrishnan and Alan Roche for helpful discussions.
We thank Chong Zhang for pointing out a gap in an earlier version,
and  referees for useful comments.
The first author was partially supported by National Science Foundation  grant DMS-1201446, National Security Agency grant H98230-16-1-0017  and PSC-CUNY. 
The second author was partially supported by Simons Foundation Collaboration Grant 240605.  The second and third authors were partially supported by National Science Foundation grant DMS-0758197.

%
%

\section{Notation}\label{notation-sec}

%
%

Either $F$ is a number field (Sections \ref{srtf-sec} and \ref{main-sec}) 
or a local field  (Sections \ref{orbint-sec} and \ref{bess-sec}), and $E$
is a quadratic \'etale extension of $F$, i.e., either a quadratic field extension
or, in Sections \ref{bess-sec} and \ref{srtf-sec}, possibly the split
algebra $E=F \oplus F$.   
 In the global (resp.\ local) case,
$\eta$ denotes the quadratic character of $F^\times\bs\A^\times$ (resp.\ $F^\times$) corresponding to $E/F$ by class field theory. 

We denote the norm map from $E$ to $F$ by $N$. For an element $\alpha \in E$ we let $\bar \alpha$ denote the image of $\alpha$ under the non-trivial element of $\Gal(E/F)$. We use the same notation for elements in $M_n(E)$. 
Denote by $I_n$ the $n\times n$ identity matrix.

We set $G' = \GL_{2n}$, viewed as an algebraic group over $F$, and we let $H'= \GL_n \times \GL_n$, which we view as a subgroup of $G'$ via the embedding
\[
(A_1, A_2) \mapsto \bmx A_1&0\\0&A_2 \emx.
\]
For $\eps \in F^\times/(NE^\times )$, set  
\[
G_\varepsilon=\left\{ \bmx \alpha & \eps \beta \\ \bar \beta & \bar \alpha \emx \in \GL_{2n}(E) : \alpha, \beta \in M_n(E) \right\}
\]
 and $H_\varepsilon$ to be the image of $\GL_n(E)$ in the block diagonal subgroup  of $G_\eps$.  
Note $G_\eps \cong G_{D_\eps} := \GL_n(D_\eps)$ where $D_\eps$ is the associated
quaternion algebra
\[ 
D_\eps=\left\{ \bmx a & \eps b \\ \bar b & \bar a \emx \in \GL_2(E) : a, b \in E \right\}.
\]
If $\eps$ is fixed, we often write  $G$ (resp. $H$) for $G_\eps$ (resp. $H_\eps$).

%
%

\section{Orbital integrals}  \label{orbint-sec}

%
%

In this section we prove the existence of matching  functions for our relative trace formulas.
A more complete matching result is now known by C.\ Zhang \cite{chong-zhang}, but
was not available at the original writing of this paper.  In addition, our approach may still
be of interest as it is more elementary and may be useful for other situations, e.g., the
archimedean case for the situation at hand.

The idea for our approach
is to translate the matching in our case to  matching between  orbital integrals over conjugacy classes on $\GL_n$ and twisted orbital integrals for $\GL_n$ over a quadratic extension. Matching in this case is known by work of Arthur and Clozel \cite{arthur:1989}, and we are able to deduce
the existence of a large class of matching functions from their work.  Throughout this section,  $F$ is a local nonarchimedean field of characteristic $0$ and $E$ is a quadratic field extension of $F$.  While working locally, we often denote the $F$-rational points
of an algebraic group $\mathbb G$ over $F$ simply by $\mathbb G$.

First we recall  results of Guo \cite{guo:1996a}  on the matching of double cosets. 
Let
\[
w = \bmx I_n&0\\0&-I_n \emx.
\]
We consider the automorphism $\theta$ of $G'$ of order 2 given by $\theta(g) =w^{-1} g w$. Then $H'$ is the set of fixed points of $\theta$. Let $S'$ be the variety
\[
S' = \left\{ g \theta(g)^{-1} : g \in G' \right\},
\]
and let $\rho : G' \to S' : g \mapsto g \theta(g)^{-1}$. The group $G'$ acts on $S'$ by twisted conjugation,
\[
g\cdot s := g s \theta(g)^{-1},
\]
so $H'$ acts on $S'$ by ordinary conjugation.
With this action we have
\[
\rho(x g h) = x \cdot \rho(g), \qquad x, g \in G', \, h \in H'.
\]
Hence $\rho$ induces an isomorphism of $G'$-spaces between the symmetric space $G'/H'$ (with $G'$ acting by left translation) and $S'$. We define
\[
\Gamma(S') = \left\{ H'\text{-conjugacy classes $[s]$ in } S' \right\},
\]
where $[s] = H' \cdot s$.
Then the set $H' \bs G' / H'$ of $H'$ double cosets in $G'$ is identified with $\Gamma(S')$. 
We set
\[
\Gamma^{\semi}(S') = \left\{ \text{semisimple } H'\text{-conjugacy classes in } S'\right\}.
\]
We define an element $s\in S'$ to be {\it ($\theta$-)regular} if $s$ is semisimple and $[s]$ in $\Gamma(S')$ has maximal dimension among the elements in $\Gamma(S')$. We denote this set by $S'^{\reg}$. We define an element $s\in S'$ to be {\it ($\theta$-)elliptic (regular)} if $s$ is regular and the centralizer of $s$ in $H'$ is an elliptic torus.  We denote this set by $S'^\ell$.  Then  we define
\[
\Gamma^{\reg}(S') = \left\{ [s] \in \Gamma(S') : s \in S'^{\reg} \right\},
\quad
\Gamma^{\ell}(S') = \left\{ [s] \in \Gamma(S') : s \in S'^{\ell} \right\},
\]
\[
G'^{\reg}=\{g \in G' : \rho(g) \in S'^{\reg}\},
\quad
\text{and}
\quad
G'^{\ell}=\{g \in G' : \rho(g) \in S'^{\ell}\}.
\]
Correspondingly, we call an $H'$ double coset of $G'$ regular or elliptic if the associated
$H'$-class in $S'$ is.
We let
$\Gamma_{H'}^{\reg}(G')$ (resp.\ $\Gamma_{H'}^{\ell}(G')$)  denote the set of regular (resp.\  elliptic) $H'$ double cosets in $G'$. We take $\Gamma^{\reg}(\GL_n(F))$ (resp.\ $\Gamma^{\ell}(\GL_n(F))$) to be the regular (resp.\ elliptic) semisimple conjugacy classes in $\GL_n(F)$.  (Note when we say regular (resp.\ elliptic), we mean
$\theta$-regular (resp.\ $\theta$-elliptic) double cosets if we are talking about $G'$ and
regular (resp.\ elliptic) in the usual sense if we are talking about $\GL_n(F)$.)  

For $A \in M_n$, let
\[
g'(A)=\bmx I_n &  A \\ I_n & I_n \emx \in G'.
\] 
By \cite[Lemma 1.3]{guo:1996a},
\begin{align} \label{eqn:reg'}
\Gamma_{H'}^{\reg}(G') = \left\{ [g'(A)] : A \in \Gamma^{\reg}(\GL_n(F)), I_n - A \in \GL_n(F) \right\}
\end{align}
and
\[
\Gamma_{H'}^{\ell}(G') = \left\{ [g'(A)] : A \in \Gamma^{\ell}(\GL_n(F)), I_n - A \in \GL_n(F) \right\},
\]
where $[g']=H'g'H'$ for $g' \in G'$.

Now we look at the double cosets on $G_\eps$ for a fixed $\eps \in F^\times$. Let $\tau \in F^\times$ such that $E=F(\sqrt{\tau})$, let
\[
w_\eps = \bmx \sqrt{\tau}I_n&0\\0&-\sqrt{\tau}I_n \emx \in G_\eps,
\]
and let $\theta_\eps$ denote the automorphism of $G_\eps$ defined by $\theta_\eps(g) = w_\eps g w_\eps^{-1}$. As before $H_\eps$ is the set of fixed points of $\theta_\eps$ and we define
\[
S_\eps = \left\{ g \theta_\eps(g)^{-1} : g \in G_\eps \right\}.
\]
Then $G_\eps$ acts on $S_\eps$ by twisted conjugation,
\[
g\cdot_\eps s := g s \theta_\eps(g)^{-1}.
\]
Define $\rho_\eps : G_\eps \to S_\eps$ by $\rho(g) = g \theta_\eps(g)^{-1}$, which identifies $G_\eps / H_\eps$ with $S_\eps$ as $G_\eps$-spaces. In particular  $H_\eps \bs G_\eps / H_\eps$ is identified with the set of $H_\eps$-conjugacy classes in $S_\eps$.

Define $\Gamma(S_\eps)$, $\Gamma^{\semi}(S_\eps)$, $\Gamma^{\reg}(S_\eps)$, 
$\Gamma^{\ell}(S_\eps)$, $S_\eps^{\reg}$ and $S_{\eps}^{\ell}$ similar to above. Let
\[
G_\eps^{\reg}=\{g \in G_\eps : \rho_\eps(g) \in S_\eps^{\reg}\}
\quad
\text{and}
\quad
G_\eps^{\ell}=\{g \in G_\eps : \rho_\eps(g) \in S_\eps^{\ell}\}.
\]
 Also define $\Gamma_{H_\eps}^{\reg}(G_\eps)$ and $\Gamma_{H_\eps}^{\ell}(G_\eps)$ similarly to the case of $G'$.

We say $g_1, g_2 \in \GL_n(E)$ are twisted conjugate if there exists $g \in \GL_n(E)$ such that $g_2=gg_1\bar g^{-1}$. Let $\Gamma^{\twist}(\GL_n(E))$ denote the set of twisted conjugacy classes in $\GL_n(E)$.   By  \cite[Lemma 1.1]{arthur:1989} there is an injective norm map
\[
\mathcal N : \Gamma^{\twist}(\GL_n(E)) \to \Gamma(\GL_n(F))
\]
defined as follows. Let $A \in \GL_n(E)$. Then $A\bar{A} \in \GL_n(E)$ is conjugate in $\GL_n(E)$ to an element $B \in \GL_n(F)$, which is unique up to conjugation in $\GL_n(F)$. One defines $\mathcal NA$ as the conjugacy class of $B$ in $\GL_n(F)$. 

We say  $g \in \GL_n(E)$ is  regular (resp.\  elliptic) twisted if $\mathcal Ng$ is regular (resp.\ regular elliptic) in $\GL_n(F)$.
Let $\Gamma^{\reg, \twist}(\GL_n(E))$ (resp.\ $\Gamma^{\ell, \twist}(\GL_n(E))$) denote the set of  regular (resp.\ elliptic) twisted  conjugacy classes in $\GL_n(E)$. 
 For $A \in M_n(E)$  let
\[
 g_\eps(A)=\bmx I_n & \eps A \\ \bar A & I_n \emx.
 \]
 Then, by \cite[Lemma 1.7]{guo:1996a},
\begin{align} \label{eqn:reg}
\Gamma_{H_\eps}^{\reg}(G_\eps) = \left\{ [g_\eps(A)] : A \in \Gamma^{\reg, \twist}(\GL_n(E)), I_n - \eps A \bar{A} \in \GL_n(E) \right\}
\end{align}
and
\[
\Gamma_{H_\eps}^{\ell}(G_\eps) = \left\{ [g_\eps(A)] : A \in \Gamma^{\ell, \twist}(\GL_n(E)), I_n - \eps A \bar{A} \in \GL_n(E) \right\}
\]
where $[g_\eps] = H_\eps g_\eps H_\eps$ for $g_\eps \in G_\eps$.

We have defined varieties $S_\eps \subset G_\eps \subset G'(E)$ and $S'\subset G' \subset G'(E)$.  By \cite[Proposition 1.3]{guo:1996a}, given a semisimple element $s \in S_\eps$ there exists $h \in H_\eps$ such that $h^{-1} s h \in S'$. This yields an embedding,
\[
\iota_\eps : \Gamma^{\semi}(S_\eps) \hookrightarrow \Gamma^{\semi}(S').
\]
According to \cite[page 117]{guo:1996a} this extends to an embedding $\Gamma(S_\eps) \hookrightarrow \Gamma(S')$.  
The map $\iota_\eps$ gives an injection of $\Gamma^{\reg}(S_\eps)$ into $\Gamma^{\reg}(S')$, and of $\Gamma^{\ell}(S_\eps)$ into $\Gamma^{\ell}(S')$.

 The injection $\iota_\eps$ induces an embedding,
\[
\iota_\eps : \Gamma_{H_\eps}^{\reg}(G_\eps) \hookrightarrow \Gamma_{H'}^{\reg}(G')
\]
by
\begin{equation} \label{eqn:iotae-def}
\iota_{\eps}([g_\eps(A)]) = [g'(\eps \mathcal NA)].
\end{equation}
and thus by restriction
$\iota_\eps : \Gamma_{H_\eps}^{\ell}(G_\eps) \hookrightarrow \Gamma_{H'}^{ \ell}(G').$

When $n$ is odd, by \cite[Lemma 1.8]{guo:1996a},
\begin{equation}\label{eqn:iotanodd}
\Gamma_{H'}^{\ell}(G')=\bigsqcup_{\eps \in F^\times/NE^\times}\iota_{\eps}\left(\Gamma_{H_\eps}^{\ell}(G_\eps)\right)
\end{equation}
and 
\begin{equation}  \label{eqn:iotanodd-reg}
\bigsqcup_{\eps \in F^\times/NE^\times}\iota_{\eps}\left(\Gamma_{H_\eps}^{\reg}(G_\eps)\right)\subset \Gamma_{H'}^{\reg}(G').
\end{equation}
When $n$ is even
\begin{equation}\label{eqn:iotaneven}
\iota_{\eps_1}\left(\Gamma_{H_{\eps_1}}^{\ell}(G_{\eps_1})\right)=\iota_{\eps_2}\left(\Gamma_{H_{\eps_2}}^{\ell}(G_{\eps_2})\right)
\end{equation}
for any $\eps_1, \eps_2 \in F^\times$.

%
%

\subsection{Local orbital integrals for $G'$}\label{local:G'}

For $g \in G'$, let
\[
H'_g=\{ (h_1, h_2) \in H'\times H' : h_1^{-1} gh_2=g\}
\]
denote the stabilizer of $g$ under the action of $H' \times H'$. We call a double coset $H'gH'$ (or
the element $g$) {\it relevant} if the map
\[
H'_g\rightarrow \C : (h_1, h_2) \mapsto \eta(\det h_2)
\]
is trivial. 

Fix a Haar measure on $\GL_n(F)$ and use this
to give $H' = \GL_n(F) \times \GL_n(F)$ the product measure.  For each double coset $H'gH'$,
we fix a (left) Haar measure on $H'_g$.
Then for $f \in C_c^\infty(G')$ and relevant $g\in G'$ we define the orbital integral
\begin{equation}\label{orbsplit}
I'_{g}(f)=\int_{H'_g \backslash H'\times H'} f(h_1^{-1}gh_2)\eta(\det h_2) \,dh_1 \,dh_2,
\end{equation}
provided it converges.  

For  $F' \in C_c^\infty(\GL_n(F))$ and $X \in \GL_n(F)$ we define the orbital integrals over conjugacy classes of $\GL_n(F)$ by
\[
O_X(F')=\int_{T'_X\backslash \GL_n(F)} F'(g^{-1}X g)\,dg,
\]
where $T'_X$ denotes the centralizer of $X$ in $\GL_n(F)$.
We will specify the Haar measure on $T'_X$ for certain $X$ in the proof of Lemma \ref{lem:redac}.

Now we will relate the orbital integrals on $G'$ to the orbital integrals on $\GL_n(F)$.

Define the open subset of $M_n$,
\[
\Ucal'=\left\{ X \in M_n : \bmx I_n & X \\ I_n & I_n \emx \in G' \right\}.
\]
Consider the mapping from $\Ucal' \rightarrow G'$ by 
\[
X \mapsto g'(X)=\bmx I_n &  X \\ I_n & I_n \emx.
\]
Given $f \in C_c^\infty(G')$ we define a smooth function $F'_f$ on $\Ucal'$ by  
\[
F'_f(X)=  \int_{(\GL_n(F))^3} f \left(\bmx A_1^{-1} & A_1^{-1} X A_2 B \\ A_2^{-1} & B \emx \right)\eta( \det (A_2B) ) \,dA_1 \,dA_2\, dB
\]
when this integral converges. 

For $F'_f(X)$ to be nice, we want to look at functions $f$ supported on the subset
\begin{equation}\label{eqn:G'main}
G'^{\main}=\left \{\bmx A & B \\ C& D \emx \in G' : A, B, C, D \in \GL_n(F)\right\}.
\end{equation}
Put
$\Ucal'^{\main}=\Ucal' \cap \GL_n(F)$  and $\Ucal'^{\ell}=\Ucal' \cap \GL_n(F)^{\ell}$,
where $\GL_n(F)^{\ell}$ denotes the set of regular elliptic elements in $\GL_n(F)$.
Note the mapping $X \mapsto g'(X)$ maps $\Ucal'^{\main}$  to $G'^{\main}$ and maps $\Ucal'^{\ell}$ to $G'^{\ell}$.

Now we prove $F'_f$ is defined and smooth on $\Ucal'^{\main}$. 
\begin{lem}\label{lem:localconv'}
Let $f \in C_c^\infty(G')$ and $X \in \Ucal'^{\main}$. Then the following statements hold:
\begin{enumerate}
\item $F'_f(X)$ is a  convergent integral. 
\item $F'_f(X)=0$ if $|\det(I_n-X)|$ is sufficiently small (in terms of an explicit constant that depends on $f$). 
\end{enumerate}
\end{lem}
\begin{proof}
We will prove that  $F_f'(X)$ is an integral over compact sets. Since $f$ is compactly supported on $G'$ there exists a compact set $\Omega_f$ in $M_n(F)$ such that if 
\[
f\left( \bmx A_1 & A_1 X A_2^{-1} B \\ A_2 & B \emx\right) \neq 0
\]
then $A_1, A_2, B, A_1XA_2^{-1}B \in \Omega_f$. It remains to prove that the determinant of each variable of integration is bounded away from zero in the support of $f$. First we note that since $\Omega_f$ is a compact subset of $M_n$ there exists a $c_f>0$ such that if $g\in \Omega_f$ then $|\det g|<c_f$. Now we note that since $f$ is compactly supported in $G'$ there exists a $c_f'>0$ such that if $f(g)\neq 0$ then $|\det g|>c_f'$. Finally we note that
\[
\bmx A_1 & A_1 X A_2^{-1} B \\ A_2 & B \emx=\bmx A_1 & 0 \\ 0& A_2 \emx \bmx I_n & X \\ I_n & I_n \emx \bmx I_n & 0 \\ 0 & A_2^{-1} B \emx .
\]
Combining these facts we see that if
 \[
f\left( \bmx A_1 & A_1 X A_2^{-1} B \\ A_2 & B \emx\right) \neq 0
 \]
  then

\[
|\det (A_1)|, \, |\det(A_2)|, \, |\det(B)|<c_f, \quad
|\det(A_1)\det(X)\det(B)|<c_f|\det(A_2)|
\]
and
\[
c_f'<|\det(I_n-X)||\det(A_1)||\det(B)|.
\]
Thus, in the support of $f$,  $|\det(A_1)|$ and $|\det( B)|$ are bounded below by $\frac{c_f'}{c_f}{|(\det(I_n-X)|^{-1}}$ and $|\det(A_2)|$ is bounded below by $\frac{(c_f')^2|\det(X)|}{c_f^3|\det(I_n-X)|^2}$. 

To prove the second statement we note that the integrand defining $F_f'(X)$ is identically zero unless $|\det(I_n-X)|> \frac{c_f'}{c_f^2}$.

\end{proof}
\begin{lem} \label{lem:redac} Let $f \in C_c^\infty(G'^{\main})$ (resp.\ $f \in C_c^\infty(G'^{\ell})$). Then
\begin{enumerate}
\item $F'_f \in C_c^\infty(\Ucal'^{\main})$ (resp.\ $F'_f \in C_c^\infty(\Ucal'^{\ell})$); and
\item  for $X \in \Ucal'^{\main}$,  $I'_{g'(X)}(f)=O_X(F'_f)$.
\end{enumerate}
\end{lem}
\begin{proof}
For the first part we note that, by the definition of $G'^{\main}$ (see \eqref{eqn:G'main}), for $f \in C_c^\infty(G'^{\main})$ there exists a compact subset $K_f$ of $\GL_n(F)$ such that  if 
\[
f\left( \bmx A_1 & A_1 X A_2^{-1} B \\ A_2 & B \emx\right) \neq 0
\]
then $A_1, A_1XA_2^{-1}B, A_2, B\in K_f$. Hence if $F'_f(X)\neq 0$ then $X \in K_f^{-1}K_fK_f^{-1}K_f$. The result for $G'^{\main}$ now follows by applying this fact and the second result from the previous lemma.   The result for $G'^{\ell}$ follows from the result for $G'^{\main}$.

We now proceed to prove the equality of orbital integrals under the mapping by a straightforward calculation. First we note that for $X \in \Ucal'^{\main}$,
\begin{equation}\label{eq:stabGprimemain}
H_{g'(X)}=\left \{\left( \bmx t & 0 \\ 0& t\emx, \bmx t & 0 \\ 0 & t \emx  \right) : t \in T'_X \right \}.
\end{equation}
Consequently, $g'(X)$ is relevant and $H_{g'(X)}$ is unimodular.
Normalize the measure on $T'_X$ so that it is compatible with this isomorphism
between $T'_X$ and $H_{g'(X)}$.
Thus
\begin{align*}
I'_{g'(X)}&=\int_{H'_{g'(X)}\backslash H'\times H'} f\left( h_1^{-1} \bmx  I_n & X \\ I_n & I_n \emx h_2 \right) \eta( \det h_2) \,dh_1\, dh_2
\\
&=\int_{H'_{g'(X)}\backslash H'\times H'} f\left( \bmx  A_1^{-1}B_1 & A_1^{-1}XB_2 \\ A_2^{-1}B_1 & A_2^{-1}B_2 \emx  \right) \eta( \det B_1B_2) \,dA_1\,dA_2 \,dB_1\,dB_2.
\end{align*}
By the change of variables $A_1 \mapsto B_1A_1$ and $B_2 \mapsto B_1A_2B_2$ the previous line equals
\begin{multline*}
\int_{T'_X\backslash \GL_n(F)} \int_{(\GL_n(F))^3} 
f\left( \bmx A_1^{-1} & A_1^{-1}B_1^{-1}XB_1 A_2B_2 \\ A_2^{-1} & B_2 \emx\right) \eta( \det (A_2B_2))\,dA_1\, dA_2\, dB_1\,dB_2
\\
= \int_{T'_X\backslash \GL_n(F)} F'_f(B_1^{-1}XB_1)\,dB_1
= O_X(F'_f).
\end{multline*}
\end{proof}

\begin{lem}\label{lem:surj}
The map from $C_c^\infty(G'^{\main})\rightarrow C_c^\infty(\Ucal'^{\main})$ defined by $f \mapsto F'_f$ is surjective.  The restriction of this map also gives a surjection $C_c^\infty(G'^{\ell})\rightarrow C_c^\infty(\Ucal'^{\ell})$.
\end{lem}
\begin{proof}
Given $\phi \in C_c^\infty(\Ucal'^{\main})$ let $\phi_0 , \phi_1 \in C_c^\infty(\GL_n(F))$ be such that
\[
\int_{\GL_n(F)}\phi_0(g)\,dg=1
\quad
\text{and} 
\quad
\int_{\GL_n(F)} \phi_1(g)\eta(\det g)\, dg =1.
\]
Then define 
\[
f\left( \bmx A_1 \\ & A_2\emx \bmx I_n &  X \\ I_n & I_n \emx \bmx I_n \\ & B\emx \right)=\phi_0(A_1)\phi_0(A_2)\phi_1(B)\phi(X).
\]
We extend $f$ to all of $G'^{\main}$ by defining $f$ to be zero on $\bmx A & B \\ C& D \emx$ if $A^{-1}BD^{-1}C \not \in \Ucal'^{\main}$. It is clear that $f\in C_c^\infty(G'^{\main})$ and  $F'_f=\phi$.

The elliptic case is similar.
\end{proof}

\subsection{Local orbital integrals for $G$}
Fix $\eps \in F^\times$ and throughout this subsection let $G=G_\eps$.

For $g \in G$, let
\[
H_g=\{ (h_1, h_2) \in H\times H : h_1^{-1} gh_2=g\}
\]
denote the stabilizer of $g$ under the action of $H \times H$. 
Fix a Haar measure $dh$ on the group $H \cong \GL_n(E)$, 
and one on each stabilizer $H_g$.
For $f \in C_c^\infty(G)$ we define the orbital integral
\begin{equation}\label{orbnonsplit}
I_{g}(f)=\int_{H_g \backslash H\times H} f(h_1^{-1}gh_2) \,dh_1 \,dh_2,
\end{equation}
when convergent.  

For  $F \in C_c^\infty(\GL_n(E))$ and $X \in \GL_n( E)$, the twisted orbital integral on $\GL_n(E)$ is
\[
TO_X(F)=\int_{T_X\backslash \GL_n(E)} F(g^{-1}X\bar g)\,dg,
\]
where $T_X$ denotes the twisted centralizer of $X$ in $\GL_n(E)$, that is,
\[
T_X=\{g \in \GL_n(E) : g^{-1} X \bar g =X\}.
\]
We specify a measure on $T_X$ for certain $X$ in the proof of Lemma \ref{lem:F_finU}.

Let
\[
G^{\main}=\left \{ \bmx \alpha & \eps \beta \\ \bar \beta & \bar \alpha  \emx \in G : \alpha, \beta \in \GL_n(E) \right \}.
\]
Consider the open subset of $M_n(E)$,
 \[
\Ucal = \Ucal_\eps=\left\{ X \in M_n(E): \bmx I_n & \eps X \\ \bar X & I_n\emx \in G \right\}.
\]
We also set $\Ucal^{\main} = \Ucal^{\main}_\eps=\Ucal \cap \GL_n(E)$ and $
\Ucal^\ell = \Ucal^{\ell}_\eps=\Ucal \cap \GL_n(E)^{\ell,\twist}$,
where $\GL_n(E)^{\ell, \twist}$ denotes the set of twisted elliptic elements of $\GL_n(E)$.
Define a mapping from $\Ucal \rightarrow G$ by 
\[
X \mapsto g(X) = g_\eps(X)=\bmx I_n & \eps X \\ \bar X & I_n \emx.
\]
Note  that this mapping restricted to $\Ucal^{\main}$ maps to $G^{\main}$. 
Now we can define a mapping of test functions.

Given $f \in C_c^\infty(G)$, we define a smooth $F_f$ on  $\Ucal$  by
\[
F_f(X)=\int_{\GL_n(E)} f\left( \bmx \alpha & 0 \\ 0 & \bar \alpha \emx g(X)\right) \, d\alpha,
\]
when this integral converges.

\begin{lem}
Let $f \in C_c^\infty(G)$ and $X \in \Ucal^{\main}$. Then 
\begin{enumerate}
\item $F_f(X)$ is a convergent integral; and
\item $F_f(X)=0$ if $|\det(I_n-\eps X\bar X)|$ is sufficiently small (in terms of an explicit constant that depends on $f$). 
\end{enumerate}
\end{lem}
The proof is very similar to, but simpler than, the proof of Lemma \ref{lem:localconv'},
so we omit it.

\begin{lem}\label{lem:F_finU}
Let $f \in C_c^\infty(G^{\main})$ (resp.\ $C_c^\infty(G^\ell)$), then 
$F_f \in C_c^\infty(\Ucal^{\main})$ (resp.\ $C_c^\infty(\Ucal^{\ell})$). 
Furthermore,  for $X \in \Ucal^{\main}$, $I_{g(X)}(f)=TO_X(F_f)$.
\end{lem}

\begin{proof} 
The first statement is similar to the case of $G'$.

For the equality of orbital integrals, first we note that for $X \in \Ucal^{\main}$
\begin{equation}
H_{g(X)}=\left\{ \left( \bmx t & 0 \\ 0 & \bar t \emx, \bmx t & 0 \\ 0 & \bar t \emx \right): t \in T_X \right \}.
\end{equation}
Similar to before, use this isomorphism with $T_X$ to transport the measure from 
$H_{g(X)}$ to $T_X$  for $X \in \Ucal^{\main}$.
Now we proceed by a straightforward calculation,
\begin{align*}
I_{g(X)}(f)=&\int_{H_{g(X)}\backslash H\times H}f(h_1^{-1} g(X) h_2) \, dh_1 \,dh_2
\\
=&\int_{T_X \backslash \GL_n(E)\times \GL_n(E)}f\left( \bmx \alpha_1^{-1} \alpha_2 & \eps \alpha_1^{-1} X \bar \alpha_2 \\ \bar \alpha_1^{-1} \bar X \alpha_2 & \bar \alpha_1^{-1} \bar \alpha_2 \emx \right) \, d\alpha_1 \, d\alpha_2.
\end{align*} 
With a change of variables sending $\alpha_1 \mapsto \alpha_2 \alpha_1$ the previous line equals
\begin{multline*}
\int_{T_X\backslash \GL_n(E)}\int_{\GL_n(E)} f\left( \bmx \alpha_1^{-1} & \eps \alpha_1^{-1} \alpha_2^{-1} X \bar \alpha_2 \\ \bar \alpha_1^{-1} \bar \alpha_2^{-1} \bar X \alpha_2 & \bar \alpha_1^{-1} \emx\right) \, d\alpha_1 \, d \alpha_2
\\
= \int_{T_X\backslash \GL_n(E)} F_f(g^{-1} X \bar g) \, dg
= TO_X(F_f).
\end{multline*}
\end{proof}

\begin{lem}\label{lem:surjeps}
The map from $C_c^\infty(G^{\main})\rightarrow C_c^\infty(\Ucal^{\main})$  defined by $f \mapsto F_f$ is surjective, and similarly for $C_c^\infty(G^{\ell})\rightarrow C_c^\infty(\Ucal^{\ell})$.
\end{lem}
\begin{proof}
Given $\phi \in C_c^\infty(\Ucal^{\main})$ let $\phi_0 \in C_c^\infty(\GL_n(E))$ be such that
\[
\int_{\GL_n(E)}\phi_0(g)\,dg=1.
\]
Then define 
\[
f\left( \bmx \alpha \\ & \bar \alpha\emx \bmx I_n & \eps X \\ \bar X & I_n \emx\right)=\phi_0(\alpha)\phi(X).
\]
It is clear that $f\in C_c^\infty(G^{\main})$ and  $F_f=\phi$. 

The elliptic case is similar.
\end{proof}

\subsection{Local matching}

Fix a set of representatives $\{ \eps_1, \eps_2\}$ for $F^\times/NE^\times$ such that $\eps_1 \in NE^\times$, $\eps_2 \not \in NE^\times$. Now we will define the notion of matching functions, for which the reader should recall the correspondence of regular
double cosets given by \eqref{eqn:iotae-def}.  
 
We also need to use compatible measures.  Namely, our orbital integrals depend upon
a choice of measures on $H$ and $H'$ as well as on stabilizers $H_{g_\eps}$ and $H'_{g'}$.
The choice of measures on $H$ and $H'$ is not important for the general notion of matching functions, as one can just scale functions appropriately.  
However, for global applications it will be convenient to assume the following: 
if $[g'] = \iota_\eps([g_\eps])$ with $g'$ regular, then $H_{g_\eps} \cong H'_{g'}$,
and we use measures compatible with this isomorphism.

\begin{defn}
Let $n$ be even and fix $\eps \in F^\times$. Let $f' \in C_c^\infty(G')$ and $f_\eps \in C_c^\infty(G_\eps)$. We say that $f'$ and $f_\eps$ are \emph{matching functions} if
\[
I'_{g'}(f')=
\begin{cases}
I_{g_\eps}(f_\eps) &  \text{if } [g']=\iota_\eps([g_\eps]) \text{ for }
[g_\eps] \in\Gamma_{H_{\eps}}^{\reg}(G_{\eps}), \\
0 & \text{if } [g'] \in \Gamma_{H'}^{\reg}(G') \setminus \iota_{\eps}\left(\Gamma_{H_{\eps}}^{\reg}(G_{\eps})\right).
\end{cases}
\]
\end{defn}

When $n$ is odd, recall the disjointedness of regular double cosets of $G'$ corresponding to
$G_{\eps_1}$ versus $G_{\eps_2}$ from \eqref{eqn:iotanodd-reg}.

\begin{defn}
Let $n$ be odd. Let $f' \in C_c^\infty(G')$ and $f_\eps \in C_c^\infty(G_\eps)$ for $\eps \in \{ \eps_1, \eps_2\}$. We say that $f'$ and $(f_\eps)_\eps$ are \emph{matching functions} if
\begin{subnumcases}{I'_{g'}(f') =}
I_{g_{\eps}}(f_{\eps})  & if $[g']=\iota_\eps([g_\eps])$ for
$[g_\eps] \in\Gamma_{H_{\eps}}^{\reg}(G_{\eps}),  \eps \in \{ \eps_1, \eps_2\},$  \\
 0  & if  $[g'] \in \Gamma_{H'}^{\reg}(G') \setminus \bigsqcup_\eps
 \iota_{\eps}\left(\Gamma_{H_{\eps}}^{\reg}(G_{\eps})\right).$ 
  \label{f-vanish-odd}
\end{subnumcases}
\end{defn}

When $n$ is odd and $f'$ matches $(f_{\eps_1},0)$ or $(0,f_{\eps_2})$, we may simply say
$f'$ matches $f_{\eps_1}$ or $f_{\eps_2}$.  We only need to consider
$f'$ matching a pair $(f_{\eps_1}, f_{\eps_2})$ for Conjecture \ref{guo-conj}(2).

We first extend the matching of orbital integrals over (twisted) conjugacy classes for $\GL_n(F)$ from \cite{arthur:1989}.   Denote by $\GL_n(F)^{\reg}$ 
(resp.\ $\GL_n(E)^{\reg, \twist}$) the set of regular elements
of $\GL_n(F)$ (resp.\ twisted regular elements of $\GL_n(E)$).  
For $\gamma \in \GL_n(F)$, denote by $[\gamma]$ its conjugacy class.
 
\begin{prop} \label{prop:AC} Fix $\eps \in \{ \eps_1, \eps_2 \}$.
\begin{enumerate}
\item
Fix $\phi \in C_c^\infty(\GL_n(E))$ (resp.\ $C_c^\infty(\GL_n(E)^{\ell, \twist})$). 
Then there exists $\phi' \in C_c^\infty(\GL_n(F))$ (resp.\ $C_c^\infty(\GL_n(F)^{\ell})$)
such that for  $\gamma  \in \GL_n(F)^{\reg}$,
\[
O_{\gamma}(\phi')=
\begin{cases}
TO_{\delta}(\phi) & \text{if } \gamma=\eps \delta \bar \delta \text{ for } \delta \in \GL_n(E),
\\
0 & \text{if } [\gamma ] \not \in \eps  \mathcal N \GL_n(E).
\end{cases}
\]

\item Suppose $n$ is odd and  $\phi' \in C_c^\infty(\GL_n(F))$ (resp.\ $C_c^\infty(\GL_n(F)^\ell)$) such that $O_\gamma(\phi') = 0$ if $[\gamma] \not \in \eps_1 \mathcal N\GL_n(E)\cup  \eps_2 \mathcal N\GL_n(E)$.
Then there exist $\phi_{\eps_1}, \phi_{\eps_2} \in C_c^\infty(\GL_n(E))$ (resp.\ $C_c^\infty(\GL_n(E)^{\ell, \twist})$) such that
\begin{equation}\label{eqn:AC2}
O_{\gamma}(\phi')=\begin{cases}
TO_{\delta}(\phi_{\eps_1}) & \text{if } \gamma =\eps_1 \delta \bar \delta \text{ for } \delta \in \GL_n(E)^{\reg, \twist},
\\
TO_{\delta}(\phi_{\eps_2}) & \text{if } \gamma=\eps_2\delta \bar \delta \text{ for } \delta\in \GL_n(E)^{\reg, \twist}.
\end{cases}
\end{equation}
\item Suppose $n$ is even and fix   $\phi' \in C_c^\infty(\GL_n(F))$ (resp.\ $C_c^\infty(\GL_n(F)^{\ell})$) such that  $O_\gamma(\phi')=0$ for $[\gamma ] \not \in \eps\mathcal N \GL_n(E)$. Then there exists $\phi \in C_c^\infty(\GL_n(E))$ (resp.\ $C_c^\infty(\GL_n(E)^{\ell, \twist})$) such that
\[
O_\gamma(\phi')=TO_\delta(\phi), \quad \text{if } \gamma=\eps\delta \bar \delta \text{ for } \delta \in \GL_n(E)^{\reg, \twist}.
\]
\end{enumerate}
\end{prop}

\begin{proof}
First we prove part (1) for $\phi \in C_c^\infty(\GL_n(E))$. We may assume $\eps_1 = 1$.
If $\eps = \eps_1$ then this is contained in Proposition 3.1 in Chapter 1 of
\cite{arthur:1989}.  Denote this matching function by $\phi_1'$. 
For  $\eps = \eps_2$ let $\phi'_2(g)=\phi'_1(\eps_2^{-1}g)$. Then $O_{g}(\phi_2')=O_{\eps_2^{-1}g}(\phi_1')$. Hence $O_{\eps_2\delta \bar \delta}(\phi_2')=TO_\delta(\phi)$ and $O_{\gamma}(\phi_2')=0$ for $[\gamma]\not \in \eps_2 \mathcal N\GL_n(E)$.

Now suppose $\phi \in C_c^\infty(\GL_n(E)^{\ell, \twist})$.  
Consider the map $s: \GL_n(F) \to F^n$
given by the coefficients of the characteristic polynomial, i.e., if $A \in \GL_n(F)$ has
characteristic polynomial $x^n + \sum c_i x^i$, put 
$s(A)=(c_0, \ldots, c_{n-1})$.  Let $\mathbf s^{\ell}$ be the image of $\GL_n(F)^{\ell}$ under $s$.
Similarly, define $s_\eps : \GL_n(E) \to F^n$ by $s_\eps(A) = s(\eps \mathcal N A)$
and $\mathbf s_\eps^{\ell} = s_\eps(\GL_n(E)^{\ell, \twist})$.  Then $s$ and $s_\eps$
are continuous and $\mathbf s_\eps^{\ell} \subset \mathbf s^{\ell}$.  We may
view the orbital integrals $\phi \mapsto TO_*(\phi)$ and $\phi' \mapsto O_*(\phi')$
as maps $C_c^\infty(\GL_n(E)^{\ell, \twist}) \to C_c^\infty(\mathbf s_\eps^{\ell})$ 
and $C_c^\infty(\GL_n(F)^\ell) \to C_c^\infty(\mathbf s^\ell)$.  These maps are surjective, which gives the desired matching in (1).

Next suppose $n$ is odd, and consider part (2) first for arbitrary $\phi' \in C_c^\infty(\GL_n(F))$. 
Write $\phi' = \phi_1' + \phi_2'$ where $\phi_i'$ has support in the set of elements whose
determinant lies in $\varepsilon_i NE^\times$.  Then the orbital integrals of $\phi'_1$
vanish off the norms, so by \cite[Proposition 3.1]{arthur:1989}, there exists $\phi_{1}$
such that $O_{\delta \bar \delta}(\phi'_1)=TO_\delta(\phi_{1})$.
Let $\~ \phi_2'(g) = \phi_2'(\eps_2 g)$, whose orbital integrals also vanish
off the norms.
Then there exists $\phi_{2}$
such that 
$O_{\eps_2 \delta \bar \delta}(\phi'_2) = O_{\delta \bar \delta}(\~ \phi'_2)=TO_\delta( \phi_{2})$, 
which is our desired matching.

For the case of elliptic support in (2),
we use the fact that the sets $\eps_1 \mathcal N\GL_n(E)^{\ell, \twist}$ and $\eps_2 \mathcal  N\GL_n(E)^{\ell, \twist}$ are disjoint and open (when regarded as subsets) in $\GL_n(F)^{\ell}$ and their union is equal to $\GL_n(F)^{\ell}$ (cf.\ \cite[proof of Lemma 1.8]{guo:1996a}).  
Then argue as in the elliptic case of (1). 

Part (3) is similar to (2) for general $\phi' \in C_c^\infty(\GL_n(F))$. 
The elliptic case of (3) is similar to that of (1), 
observing that the vanishing condition implies the orbital integral map 
$\phi' \mapsto O_*(\phi')$ 
gives a function in $C_c^\infty(\mathbf s^{\ell})$ with support in $\mathbf s_\eps^{\ell}$.
\end{proof}

Now we deduce our local matching results, both for functions with support
in ``main'' sets and in the elliptic sets.

\begin{prop}\label{prop:matchGtoG'}
Fix $\eps \in F^\times$ and  $f \in C_c^\infty(G_\eps^{\main})$ (resp.\ $C_c^\infty(G_\eps^\ell)$). Then there exists $f' \in C_c^\infty(G'^{\main})$ (resp.\ $C_c^\infty(G'^\ell)$) such that 

\[
I'_{g'(\gamma)}(f')=
\begin{cases}
I_{g_{\eps}(\delta)}(f) &  \text{if } \gamma=\eps \delta \bar \delta \text{ for } \delta \in \Ucal_\eps^{\main},
\\
0 & \text{if } [\gamma ] \not \in \eps \mathcal N\Ucal_\eps^{\main}.
\end{cases}
\]
 In particular, $f'$ matches with $f$.
\end{prop}
\begin{proof} 
The arguments for $G_\eps^{\main}$ and $G_\eps^\ell$ are identical, except for the
use of  different parts of
Proposition \ref{prop:AC}.  We just write the argument down for $G_\eps^{\main}$.
 
Let $f \in C_c^\infty(G_\eps^{\main})$. Then by Lemma \ref{lem:F_finU}, $F_f \in C_c^\infty(\Ucal^{\main})$ and for all $X \in \Ucal^{\main}$, $I_{g(X)}(f)=TO_X(F_f)$. By Proposition \ref{prop:AC} there exists a $\phi' \in C_c^\infty(\GL_n(F))$ such that for $X \in \GL_n(E)^{\reg, \mathrm{tw}}$,
$TO_X(F_f)=O_{\eps X\bar X}(\phi')$ and 
$O_{Y}(\phi')=0$ for $Y \not \in \eps \mathcal N \GL_n(E)$. By Corollary 3.13 of Chapter 1 in
\cite{arthur:1989}, these orbital integrals are equal up to a sign for any $X \in  \GL_n(E)$.
Since $F_f\in C_c^\infty(\Ucal^{\main})$, there exists a  $c$ such that $F_f(X)$ and hence also $TO_X(F_f)$ vanish for $X$ such that $|\det(I_n-\eps X \bar X)|>c$. Thus $O_{X'}(\phi')=0$ unless $|\det(I_n-X')|>c$. Let $\1_c$ be the characteristic function of 
\[
\{ X' \in \Ucal'^{\main} : |\det(I_n-X')|>c\}
\]
and set $\tilde \phi'=\phi' \cdot \1_c$. Then $\tilde \phi' \in C_c^\infty(\Ucal'^{\main})$ and $O_{X'}(\tilde \phi')=O_{X'}(\phi')$ for all $X' \in \Ucal'^{\main}$. By Lemma \ref{lem:surj} there exists an $f' \in C_c^\infty(G'^{\main})$ such that $\tilde \phi'=F_{f'}$ and $I'_{g'(X')}(f')=O_{X'}(\tilde\phi')$.
\end{proof}

We also want converse matching results.  

\begin{prop}\label{prop:matchG'toG}
Let $n$ be odd and let $f' \in C_c^\infty(G'^{\main})$ (resp.\ $C_c^\infty(G'^{\ell})$) 
satisfying the vanishing condition \eqref{f-vanish-odd}.
Then there exist $f_\eps \in C_c^\infty(G_\eps^{\main})$ (resp.\ $C_c^\infty(G_\eps^{\ell})$) for 
$\eps=\eps_1, \eps_2$ such that $(f_\eps)_\eps$ and $f$ are matching. 
\end{prop}

Note  \eqref{f-vanish-odd}
is vacuously satisfied when $f'$ has elliptic support by \eqref{eqn:iotanodd}.

\begin{proof} Let $f' \in C_c^\infty(G'^{\main})$.  By Lemma \ref{lem:redac},
we can consider $\phi' = F'_{f'} \in C_c^\infty(\GL_n(F))$.   Then apply Proposition 
\ref{prop:AC} to get the existence of $\phi_{\eps_1}, \phi_{\eps_2} \in C_c^\infty(\GL_n(E))$ 
that satisfy \eqref{eqn:AC2}. We now apply Lemma \ref{lem:surjeps} and complete the proof 
as before to find $f_{\eps_1}, f_{\eps_2}$.  The elliptic case is similar.
\end{proof}

\begin{prop}\label{prop:matchG'toGeven}
Let $n$ be even, $\eps \in F^\times$ and $f' \in C_c^\infty(G'^{\main})$ (resp.\ $C_c^\infty(G'^{\ell})$) such that $I'_{g'(X)}(f')=0$ for $X \not \in\eps \mathcal  N \GL_n(E) $. Then there exists $f \in C_c^\infty(G_\eps^{\main})$ (resp.\ $C_c^\infty(G_\eps^{\ell})$) such that $f$ and $f'$ are matching.
\end{prop}
\begin{proof}
This proof is similar to the proof of the previous proposition.
\end{proof}

\begin{cor}\label{cor:match}
Let $n$ be even, fix $i \in \{ 1, 2 \}$, and let $f_{\eps_i}\in C_c^\infty(G_{\eps_i}^{\ell})$.
Then for $j \in \{ 1, 2 \}$, there exists $f_{\eps_j} \in C_c^\infty(G_{\eps_j}^{\ell})$ such that $f_{\eps_i}$ and $f_{\eps_j}$ are matching.
\end{cor}

Here $f_{\eps_1}$ matching $f_{\eps_2}$ means they both match a single 
$f' \in C_c^\infty(G')$.  We do not know an analogue of this corollary for the main or regular sets, which
is what forces us to make elliptic assumptions in Theorems \ref{main-thm3}
and \ref{local-thm2}.

\begin{proof} This result follows from Propositions \ref{prop:matchGtoG'} and \ref{prop:matchG'toGeven},  and \eqref{eqn:iotaneven}.
\end{proof}

\begin{rem} \label{rem:arch-ell}
The above matching results for the elliptic set can be carried out similarly at archimedean places.
\end{rem}

By work of Guo \cite{guo:1996a} the fundamental lemma is known for the unit element in the Hecke algebra. As we will need this result for the global comparison, we state it here.

\begin{prop}\cite{guo:1996a}\label{prop:fund}
Let $E/F$ be an unramified quadratic extension of local nonarchimedean fields with odd residual characteristic.  Let $\Xi'$ be the characteristic function of $G'(\O)$ and $\Xi_{\eps_1}$ the characteristic function of $G_{\eps_1}(\O)$, where $\O$ is the integer ring of $F$. 
Assume the measures on $H$ and $H'$ are such that $\vol(H(\O)) = \vol(H'(\O))$. 
Then $\Xi' $ and $\Xi_{\eps_1}$ are matching functions.
\end{prop}

%
%

\section{Local Bessel distributions} \label{bess-sec}

%
%

In this section, $F$ is a local field of characteristic
zero (possibly archimedean), $E$ is a quadratic \'etale extension of $F$ (possibly $F \oplus F$), and $D=D_\eps$ is a
 quaternion algebra of $F$ which splits over $E$.  We allow for the possibility that $D$ is split, i.e.,
 $G := G_\eps = G'$.  Further, if $F$ is archimedean, we assume $E/F$ is split.

Let $\pi$ be an irreducible admissible unitary representation of $G$ with trivial central character.  
Let $\lambda$ be an $H$-invariant linear form on $\pi$.   Since $(G,H)$ is a Gelfand pair
(\cite{jacquet:1996} for $E/F$ split nonarchimedean, \cite{aizenbud} for $E/F$  split archimedean, 
and \cite{guo:1997} for $E/F$  inert nonarchimedean), $\lambda$ is unique up to scaling.
If $\pi$ has a nonzero $H$-invariant linear form, we say $\pi$ is $H$-distinguished.

We define the {\em local Bessel distribution} on $G$ for $\pi$ with respect to $\lambda$ to be
\[ B_\pi(f) = \sum_\phi \lambda(\pi(f) \phi) \overline{\lambda(\phi)} \]
for $f \in C_c^\infty(G)$, where $\phi$ runs over an orthonormal basis for $\pi$.
  Note $B_\pi \equiv 0$ if and only if $\lambda$ is zero.  Local Bessel distributions are also sometimes
  referred to as {\em spherical characters} in the literature.
  
Now let $\pi'$ be an irreducible admissible unitary representation of $G'$ with trivial central character.
If $E/F$ is split, we may identify $G$ with $G'$,  $H$ with $H'$, and define the local Bessel distribution
on $G'$ for $\pi'$ as above.  Assume $E/F$ is inert.  
Let $\lambda_1$ be an $H'$-invariant linear form on $\pi$ and let $\lambda_2$ be 
an $(H',\eta \circ \det)$-equivariant linear form on $\pi'$.  Since $(G',H')$ is a Gelfand pair, 
$\lambda_1$ and $\lambda_2$ are unique up to scaling (note $\lambda_2$ is the
same as an $H'$-invariant linear form on $\pi' \otimes \eta$).  We define the
{\em local Bessel distribution} on $G'$ for $\pi'$ with respect to $(\lambda_1,\lambda_2)$ to be
\[ B_{\pi'}(f') = \sum_\phi \lambda_1({\pi'}(f') \phi) \overline{\lambda_2(\phi)}, \]
where $\phi$ runs over an orthonormal basis for $\pi'$ and $f' \in C_c^\infty(G')$.  
As before, $B_{\pi'} \equiv 0$ if and only if $\lambda_1$ or $\lambda_2$ is zero.

\subsection{Generalities}

We will now establish some results we will need about $B_\pi(f)$ and $B_{\pi'}(f')$.   

Note that, by definition, the distribution $B_\pi(f)$ is bi-$H$-invariant. 
Similarly, if $E/F$ is inert, the distribution $B_{\pi'}(f')$
is left $H'$-invariant and is right 
$(H',\eta \circ \det)$-equivariant.  

From now on we assume $F$ is nonarchimedean and that $B_\pi \not \equiv 0$.  We say a distribution $B$ on $G$ is locally integrable if there is a locally integrable function $b$ on $G$
such that $B(f) = \int_G f(g)b(g) \, dg$ for all $f \in C_c^\infty(G)$.

\begin{prop} \label{Bpi-supp-prop}
Suppose the residual characteristic of $F$ is odd.
Then $B_\pi(f)$ is locally integrable.  In particular,
for any dense open subset $X \subset G$, there exists
$f \in C_c^\infty(X)$ such that $B_\pi(f) \ne 0$.
\end{prop}

This is a minor extension of \cite{guo:spherical}. 

\begin{proof}
This result was proved in \cite{guo:spherical} under the  additional hypotheses that $G$ is split.  The proof of \cite{guo:spherical} in the case $D$ is ramified goes through similarly.  We
outline this now.  
We remark Rader--Rallis \cite{rader:rallis} showed (in great generality) that $B_\pi(f)$ is 
 locally integrable on $G^{\mathrm{reg}}$.

We drop the $\eps$ subscript from the notation in the previous section, e.g., $S=S_\eps$.
Fix $s \in S$ semisimple and let $x \in \rho^{-1}(s)$. 
Let $G_s$ be the connected component of the stabilizer of $s$ in $G$, and $H_s = G_s \cap H$.  Let $U_x$ be the
set of $g \in G_s$ such that the map $H \times G_s \times H \to G$ given by $(h_1,g,h_2) \mapsto h_1 xgh_2$ is submersive
at $(1,g,1)$.  This is an open bi-$H_s$-invariant neighborhood of 1 in $G_s$.   Further the image $\Omega_x$ of $U_x$ is an
open bi-$H$-invariant neighborhood of $x$ in $G$.  By standard Harish-Chandra theory, the restriction of $B_\pi$ to $C_c^\infty(\Omega_x)$ 
may be viewed as an $H_s$-invariant distribution $\Theta_x$ on $U_x$. 

For rather general symmetric spaces,
Rader--Rallis \cite{rader:rallis} proved a germ expansion theorem
for spherical characters when $x=1$, which was extended to arbitrary $x$ by Guo \cite[Theorem 2.1]{guo:spherical}.  This germ expansion expresses
$\Theta_x$, in a neighborhood of 1 in $U_x$ as a linear combination of Fourier transforms $\hat \Lambda$
of $H_s$-invariant distributions $\Lambda$ on $\mathfrak s_s$
supported in $\mathfrak N_{\mathfrak s_s}$.  Here $\mathfrak s_s$ is the Lie algebra of $S_s = G_s/H_s$, and $\mathfrak N_{\mathfrak s_s}$ is
the subset of nilpotent elements.  This reduces the problem to showing the $\hat \Lambda$'s are locally integrable.

Note the Lie algebra of $G$ can be written as 
\[ \mathfrak g =\left \{ \bmx \alpha & \eps \beta \\ \bar \beta & \bar \alpha \emx : \alpha, \beta \in M(n,E)\right \}.\]  
Consider the subspaces 
\[ \mathfrak h =\left \{ \bmx \alpha & \\ & \bar \alpha \emx \right \}, \quad \text{and} \quad
\mathfrak s = \left\{ \bmx & \eps \beta \\
\bar \beta & \emx\right \}. \]
Note $\mathfrak h$ is the Lie algebra of $H$, and $\mathfrak s$ plays the role of the Lie
algebra for $S$.

In the case that $D$ is split, Guo \cite[Proposition 4]{guo:spherical}
shows that the representation $(H_s, {\mathfrak s_s})$ is isomorphic to a product of representations of the form
$(G_0, \mathfrak g_0)$ and $(H(n_i), \mathfrak s(n_i))$.  Here, $G_0$ is a certain reductive group, and 
$H(n_i)$ and $\mathfrak s(n_i)$ denote the corresponding $H$ and $\mathfrak s$ for $G(n_i) = \GL(n_i,D)$.
It is not hard to show the same statement is true when $D$ is nonsplit (cf.\ \cite[Proposition 4.7]{chong-zhang}).

Harish-Chandra showed each nilpotent orbit in $\mathfrak g_0$ has a $G$-invariant measure with locally integrable Fourier
transform.  
To complete the proof, one needs to show the analogous statement for pairs $(H, \mathfrak s)$, i.e.,
each nilpotent orbit in $\mathfrak s$ has an $H$-invariant measure with locally integrable Fourier transform.
Guo achieves this by proving certain integral formulas for distributions $\Lambda$ on $\mathfrak s$
given by nilpotent orbital integrals and their Fourier transforms $\hat \Lambda$, and showing that
$\hat \Lambda$ is locally integrable using a Weyl integration formula.

Since the representations $(H,\mathfrak s)$ are isomorphic in the cases where
$D$ is split and where $D$ is ramified (the action is given by twisted conjugation of $\GL(n,E)$ on
$M(n,E)$), and the description of the nilpotent orbits of $\mathfrak s$ is the same in both cases
(cf.\ \cite{guo:1997}), Guo's proof extends to the case where $D$ is ramified without difficulty.
\end{proof}

\begin{lem}\label{Bpi-Gpm-lem} Put $G^\pm = \{ g \in G : \eta(\det g) = \pm 1 \}$.
Suppose $\pi \not \cong \pi \otimes \eta$ (hence $E/F$ is not split).  
Then $B_\pi$ is neither supported on
$G^+$ nor $G^-$. 
\end{lem}

\begin{proof} Note $B_\pi$ and $B_{\pi \otimes \eta}$ are linearly independent
(cf.\ \cite[Lemma 2.2]{FLO}).  For $f \in C_c^\infty(G)$, put $f^\eta(g) = \eta(\det g)f(g)$. Note that  $\pi(f^\eta)=(\pi\otimes \eta)(f)$. Thus $B_\pi(f^\eta) = \kappa B_{\pi \otimes \eta}(f)$ for some 
$\kappa \in \C$ where $\kappa=1$ if the same $H$-invariant linear form $\lambda$ is chosen for both $B_\pi$ and $B_{\pi\otimes \eta}$. However if $f \in C_c^\infty(G^+)$, then $f^\eta = f$.  Hence if $B_\pi$
is supported on $G^+$, we would have $B_\pi(f) = B_\pi(f^\eta) = \kappa B_{\pi \otimes \eta}(f)$
for all $f \in C_c^\infty(G)$, a contradiction.  The case of $G^-$ is similar.
\end{proof}

\begin{lem}  \label{Bpi-pm-lem}
Suppose $\pi \not \cong \pi \otimes \eta$ and $F$ has odd residual characteristic.  Then
for any open dense $X \subset G$ and any $c \in \C^\times$, 
there exists $f \in C_c^\infty(X)$ such that
$B_\pi(f) \ne c B_{\pi \otimes \eta}(f)$.
\end{lem}

\begin{proof}
Put $X^\pm = X \cap G^\pm$.  By Proposition \ref{Bpi-supp-prop} and Lemma \ref{Bpi-Gpm-lem}, we know there
exist $f_\pm \in C_c^\infty(X^\pm)$ such that $B_\pi(f_\pm) \ne 0$.  Since
$B_\pi(f_\pm) = \pm B_{\pi \otimes \eta}(f_\pm)$, we can choose constants $c_\pm$ such that
$f=c_+ f_+ + c_-f_-$ satisfies the desired property.
\end{proof}

\subsection{Elliptic support of Bessel distributions} \label{ell-supp-sec}
For use in our simple trace formula, we in fact want to know something stronger
about our Bessel distributions $B_\pi$---that they often do not vanish on the $H$-elliptic
set for $H$-distinguished $\pi$.  Namely, we will say $\pi$ (not necessarily
a priori $H$-distinguished) is {\it $H$-elliptic} if there exists an $H$-invariant functional $\lambda$
on $\pi$ such that the associated Bessel distribution $B_\pi(f) \ne 0$ for some 
$f \in C_c^\infty(G^{\mathrm{ell}})$.  In this section, we give a local proof of the following.

\begin{prop} \label{ell-supp-prop}
There exist $H$-elliptic (simple) supercuspidal representations when $E=\Q_2\oplus \Q_2$.
\end{prop}

We show this by reducing the problem to showing the nonvanishing of an elliptic orbital
integral for a supercusp form.  In the next section, we will use this
result together with global methods to show
the existence of  $H$-elliptic  representations for a general $E$.

In Section \ref{orbint-sec}, we defined local orbital integrals $I_g(f)$ for
$f \in C_c^\infty(G)$, which converge for $g \in G^{\mathrm{ell}}$.  Here it is more
convenient to work with orbital integrals for functions $\Phi \in C_c^\infty(G/Z)$,
for which we consider the orbital integral
\[ I^Z(g; \Phi) = \int_{H/Z} \int_{H/Z} \Phi(h_1 g h_2) \, dh_1 \, dh_2. \]
(Here we do not bother to quotient out by $H_g$, which has finite volume for
elliptic $g$.)
Note any such $\Phi$ is of the form $\Phi(g) = \int_Z f(gz) \, dz$ in which case we have
\[ I^Z(g; \Phi) = c I_g(f) \]
for some $c \ne 0$.  Hence $I^Z(g;\Phi)$ converges for $g \in G^{\mathrm{ell}}$ and is nonzero
if and only if $I_g(f)$ is nonzero.  On $G^{\mathrm{ell}}$, $I^Z(g; \Phi)$ is a smooth function.

\begin{lem} \label{ell-int-lem}
Suppose $\pi$ is a supercuspidal representation of $G$ over a $p$-adic field $F$.
Then $\pi$ is $H$-elliptic if and only if
the orbital integral $I^Z(g;\Phi) \ne 0$ for some $g \in G^{\mathrm{ell}}$ and some matrix
coefficient $\Phi$ of $\pi$.
\end{lem}

\begin{proof}
Let $\Phi$ be any matrix coefficient of $\pi$.  Then \cite[Theorem 6.1]{murnaghan} tells us
\[ D_\Phi(f) = \int_{H/Z} \int_{H/Z} \int_G f(g) \Phi(h_1 g h_2) \, dg \, dh_1 \, dh_2 \]
defines a bi-$H$-invariant distribution on $G$.  For $f$ supported on $G^{\mathrm{ell}}$,
we have absolute convergence of the orbital integrals and can write
\begin{equation} \label{eq:DPhi}
D_\Phi(f) = \int_G f(g) I^Z(g; \Phi) \, dg.
\end{equation}

Suppose $\pi$ is $H$-elliptic.  Then by \cite[Corollary 1.11]{czhang:scnote}
$B_\pi =  D_{\Phi}$ for some matrix coefficient $\Phi$ of $\pi$ (this also follows from
Theorem 6.1 and Lemma 8.3 of \cite{murnaghan} when $G= \GL(2n)$).
Since $B_\pi \ne 0$, $D_{\Phi} \ne 0$.
Then \eqref{eq:DPhi} implies $I^Z(g; \Phi) \ne 0$ for some
$g \in G^{\ell}$.

Now suppose $I^Z(g; \Phi) \ne 0$ for some $g \in G^{\mathrm{ell}}$.
Since $I^Z(g; \Phi)$ is locally constant on $G^{\mathrm{ell}}$,
we may choose $f$
with small support around $g$ to get $D_\Phi(f) \ne 0$.  In particular,
$D_\Phi \not \equiv 0$, so we must have that $\pi$ is $H$-distinguished
(cf.\  \cite[Theorem 6.1]{murnaghan}) and, by the uniqueness of $H$-invariant linear forms, $D_\Phi = c B_\pi$ for some
nonzero $c$.  Hence $B_\pi(f) \ne 0$.
\end{proof}

Now we will show the existence of $H$-elliptic supercuspidal representations. 
Let us assume $F$ is nonarchimedean, 
$E= F \oplus F$, so $G = \GL(2n)$ and $H = \GL(n) \times \GL(n)$.

We first recall some facts about simple supercuspidal representations of $G$.  See
\cite{KL-ssc} for more details.  

Let $\O$ be the integers of $F$ with maximal ideal $\p = \varpi \mathcal O$, 
and residue field $\F_q = \O/\p$ of order $q$.
Let $K=K_{2n}$ be the subgroup of $G(\mathcal O)$ consisting
of matrices which are upper unipotent mod $\mathfrak p$, and let $J = ZK$.  Fix a nontrivial
character $\psi$ of $\F_q = \O/\p$ and $t_1, \ldots, t_{2n} \in \F_q^\times$.  Now define a character
$\chi$ of $J$ by $\chi|_Z = 1$ and
\[ \chi(k) = \psi( t_1 x_1 + t_2 x_2 + \cdots + t_{2n} x_{2n} ) \]
for $k \in K$ of the form
\[ k \equiv \bmx 1& x_1 & * & \cdots & * \\
0 & 1 & x_2 & * & * \\
0 & 0 & \ddots & \ddots & \vdots \\
\vdots & & \ddots& 1 & x_{2n-1} \\
\varpi x_{2n} & 0 & \cdots &0 & 1
\emx \mod \p, \]
i.e., if $k = (k_{ij})_{ij}$, then $x_i = k_{i+1,i}$ for $1 \le i \le 2n-1$ and 
$\varpi x_{2n} = k_{1,{2n}}$.  

Let $\pi_\chi = c\mathrm{-Ind}_J^G \chi$.  This is a direct sum of $2n$ 
irreducible supercuspidal representations $\pi_{\chi,1}, \ldots, \pi_{\chi,2n}$,
which are the simple supercuspidals associated to $\chi$.  Next we 
want to define a matrix coefficient $\Phi$ of $\pi_\chi$.  This will be a sum
of matrix coefficients for the simple supercuspidals $\pi_{\chi,i}$.

For our purposes, we call $A \in M_m(F)$ a permutation matrix
if it has exactly one nonzero entry in each row and column.  If $e_1, \ldots, e_m$
denotes the standard basis of $F^m$, then $A$ permutes the lines $Fe_i$, and
we think of $A$ as representing the element of $\sigma_A \in S_m$, the
symmetric group on $\{ 1, 2, \ldots, m \}$, given by
$\sigma_A(i) = j$ if $A \cdot Fe_i = Fe_j$.

Let $w = w_{2n} \in \GL_{2n}(F)$ be the 0--1 permutation matrix associated to 
the product of 2-cycles $(2i, \ 2n-2i+1)$ for $1 \le i \le \lfloor \frac n2 \rfloor$.
Note we can inductively define $w_{2n}$ by
\[ w_2 = \bmx 1 &\\ & 1 \emx, \quad w_4 = \bmx 1 &&& \\ &&1& \\ &1&& \\ &&& 1 \emx,
\quad  \text{and} \quad w_{2n} =  \bmx 1&0&&0&0 \\ 0&0&&1&0 \\ &&\boxed{w_{2n-4}}&& \\  0&1&&0&0 \\ 0&0&&0&1 \emx \]
for $n > 2$.  Consider the conjugate subgroups $K' = K'_{2n} = w_{2n}K_{2n}w_{2n}$ and 
$J' =ZK'$.  For
$g \in J'$, let $\chi'(g) = \chi(wgw)$.  Let
$ \Phi(g) = 1_{J'}(g) \chi'(g),$
which is a matrix coefficient for $\pi_\chi$.  The reason to work with this $\Phi$ rather than
$\Phi_J(g) = 1_{J}(G)\chi(g)$ is that $\chi |_{H \cap J}$ is nontrivial, which
forces the integrals $I^Z(g; \Phi_J)$ to vanish (when convergent).  

To see $\chi' |_{H \cap J'} = 1$,  we can inductively write the subgroups
$K'_{2n}$ of $\GL_{2n}(\mathcal O)$ as
\[ K'_2 = \{ k \equiv \bmx 1 & x_1 \\ \varpi x_2 & 1 \emx \mod \p \}, \qquad
K'_4 = \{ k \equiv \bmx 1 & * & x_1 & *  \\ 
0 & 1 & 0 & x_3 \\
0 & x_2 & 1 & * \\
\varpi x_4 & 0 & 0& 1 \emx \mod \mathfrak p \}, \]
and, for $n > 2$, 
\[ K_{2n}' =  \{ k \equiv
\left( \begin{array} {c|c|c}
\begin{matrix} 1& * \\ 0 & 1
\end{matrix}&
\begin{matrix} * & * & \cdots & * \\ 0 & 0 & \cdots & 0
\end{matrix}&
\begin{matrix} x_1 & * \\ 0 & x_{2n-1}
\end{matrix} \\
\hline
\begin{matrix} 0 & * \\ \vdots & \vdots  \\ 0 & * \\ 0 & x_{2n-2} 
\end{matrix} &
\begin{matrix} w_{2n-4} k_{2n-4} w_{2n-4}
\end{matrix} &
\begin{matrix}  0 & * \\ \vdots & \vdots  \\ 0 & * \\ 0 & *
\end{matrix} \\
\hline
\begin{matrix} 0 & * \\ \varpi x_{2n} & 0
\end{matrix} &
\begin{matrix} x_2 & * & \cdots & * \\ 0 & 0 & \cdots & 0
\end{matrix} &
\begin{matrix} 1 & * \\ 0 & 1
\end{matrix}
\end{array} \right) \mod \p  : k_{2n-4} \in K_{2n-4}\},  \]
where we write $k_{2n-4}$ in the form
\[ k_{2n-4} \equiv 
\bmx 1  & x_3 & * & \cdots & * \\
0 & 1 & x_4 & \cdots & * \\
\vdots & & \ddots & \ddots & \vdots \\
0 &  &   & 1 & x_{2n-3} \\
0 & 0 & \cdots & 0 & 1
\emx \mod \p.  \]
As before $x_i$'s and $*$'s denote arbitrary elements of $\O$, and with $k$
in the above form, we have
\[ \chi'(k) = \psi(t_1 x_1 + t_2 x_2 + \cdots + t_{2n} x_{2n}). \]
It is now clear that all the $x_i$'s appear in the upper right (for $i$ odd)
or lower left (for $i$ even) $n \times n$ block of $k$, so $\chi'|_{H \cap J'} = 1$.

Let $k_0 \in K'_{2n}$ be the element where all $x_i$'s and diagonal entries are 1,
and all other entries are 0.  We will now show that $k_0 \in G^{\mathrm{ell}}$.  
Write $k_0 = \bmx I_n & X \\ Y & I_n \emx$ where all four blocks are $n \times n$.
It is easy to see that $k_0$ is ($\theta$-)elliptic if and only if $XY \in \GL_n(F)^{\mathrm{ell}}$.
The inductive description of $K'_{2n}$
implies that $X$ and $Y$ are permutation matrices, hence so is $XY$.

We claim that $XY$ represents an $n$-cycle in $S_n$.  Since exactly one
entry is $\varpi$ and the others are 1, this will imply $XY$ is similar to the matrix
\[ \bmx 
0 & 1  && \\
& \ddots & \ddots \\
&&0&1 \\
\varpi &&& 0 \emx, \]
whence has characteristic polynomial $\lambda^n -(-1)^n \varpi$, and 
must therefore be elliptic.

First we show the claim for $n$ odd.  
In this case, $X=Y$ so $XY=X^2$, so it suffices to check that $X$ is an 
$n$-cycle.  Note that for $1 \le j \le n$,
\[ Xj = \begin{cases}
n & j=1 \\
{n-j} & j \text{ even} \\
{n-j+2} & j \text { odd} > 1.
\end{cases} \]
Consequently, we see $X$ represents the permutation 
\[ (1, \ n, \ 2, \ (n-2), \ 4, \ (n-4), \ \cdots , \ (n-1) ) \]
which has order $n$.

Now suppose $n=2m$.  Then 
\[ Xj = \begin{cases}
n-j & j \text{ odd} \\
n-j+2 & j \text{ even}
\end{cases} 
\quad
\text{and}
\quad
Yj = \begin{cases}
n & j = 1\\
1 & j = n \\
n-j & j < n \text{ even } \\
n-j+2 & j > 1 \text{ odd}. 
\end{cases} \]
Consequently, $XY$ represents the $n$-cycle
\[ (1, \ 2, \ 4, \  6, \ \cdots , \ n, \ (n-1), \ (n-3), \ (n-5), \ \cdots , \ 3 ). \]
Thus our claim is justified, and $k_0$ is indeed elliptic.

\begin{lem} Let $h_1, h_2 \in H$ such that $h_1 k_0 h_2 \in K'$.
Then $\chi'(h_1 k_0 h_2) = \psi(u_1 t_1  + \cdots + u_{2n} t_{2n} )$
for some units $u_i \in \O^\times$.
\end{lem}

\begin{proof} Write $h_1 = \bmx A_1 & \\ & A_2 \emx$ and $h_2 = \bmx B_1^{-1} & \\ & B_2^{-1} \emx$.
Then
\[ h_1 k_0 h_2 = \bmx A_1 B_1^{-1} & A_1 X B_2^{-1} \\ A_2 Y B_1^{-1} & A_2 B_2^{-1} \emx. \]
By the structure of $K'$, we need $A_1 B_1^{-1}, A_2B_2^{-1} \in \GL_n(\O)$.  For $A \in M_n(F)$,
write $v(A)$ for $v(\det A)$, where $v$ is the valuation.  Then $v(A_1) = v(B_1)$ and
$v(A_2) = v(B_2)$.  On the other hand for $h_1 k_0 h_2 \in K'$,
looking at the upper right and lower left blocks of $h_1 k_0 h_2$, we need $v(A_1 X B_2^{-1}) \ge 0$
and $v(A_2 Y B_1^{-1}) \ge 1$.  Since $v(X) = 0$ and $v(Y) = 1$, this means
$v(A_1) \ge v(B_2) = v(A_2)$ and 
$v(A_2) \ge v(B_1) = v(A_1)$, i.e., $v(A_1) = v(A_2) = v(B_1) = v(B_2)$.  
Thus the upper right block of $h_1k_0h_2$ must have valuation 0,
and the lower left block must have valuation 1.

Write $h_1 k_0 h_2 = \bmx U & X' \\ Y' & V \emx$.  Now we can use the inductive expression for
$K'_{2n}$ to see that $v(X') = 0$ and $v(Y') = 1$ implies that all of the $x_i$-coordinates for
$h_1 k_0 h_2$ must be units in $\O$.  This is evident for $x_1, \ldots, x_{2n-1}$; for $x_{2n}$,
take the determinant of $Y'$ by expanding in cofactors along the first column (or last row).
\end{proof}

\begin{cor} Suppose $q=2$.  Then $\chi'(g) =1$ for any $g \in J' \cap Hk_0H$.
Consequently, $I^Z(k_0; \Phi) \ne 0$.
\end{cor}

\begin{proof} When $q=2$, any element of $\mathcal O^\times$ is $1$ mod $\p$,
which means, with notation as in the lemma above, if $g=zh_1k_0h_2$ with $z \in F^\times$,
then $\sum u_i t_i  \equiv 0 \mod \p$.
This proves the first statement, which means $I^Z(k_0; \Phi)$ is just the
(nonzero finite) volume of $\{ (h_1, h_2) \in H/Z \times H/Z : h_1 k_0 h_2 \in J' \}$.
\end{proof}

Since $\Phi$ is a sum of matrix coefficients for $\pi_{\chi,1 }, \ldots, \pi_{\chi,2n}$,
this means that when $q=2$ (so we have no choice of $t_i$'s and there is a unique
$\chi$), one of the simple supercuspidals $\pi_{\chi,1}, \ldots, \pi_{\chi,2n}$ is $H$-elliptic,
completing the proof of Proposition \ref{ell-supp-prop}.

%
%

\section{A simple relative trace formula} \label{srtf-sec}

%
%

Return to the global situation, i.e., $F$ is a number field and $D$ a quaternion algebra
over $F$ which splits over a quadratic \'etale extension $E/F$. 
We allow for the possibility that $E=F \oplus F$ when $D=M_2(F)$, i.e.\ $G$ and $H$ may be $G'$ and $H'$, in order to treat both trace formulas of interest simultaneously.

Let $\theta$ be the inner automorphism of $G$ fixing $H$ as defined locally.
The notions of ($\theta$-)regular and  ($\theta$-)elliptic
elements of $G(F)$ are defined similarly as in the local case.  
If $E$ is a field, put $\chi(h) = \eta(\det h)$.  Otherwise put $\chi(h) = 1$.

We choose Haar measures $dz= \prod dz_v$ and $dh = \prod dh_v$  on
$Z(\A)$ and $H(\A)$ such that at all finite places
 $Z(\mathcal O_v)$ and $H(\mathcal O_v)$ have volume 1.  

For $f \in C_c^\infty(G(\A))$, we define the kernel
\[ K(x,y) = \sum_{ \gamma \in Z(F) \bs G(F)} \int_{Z(\A)} f(zx^{-1} \gamma y) \, dz, \]
and consider the (partial) distribution given by
\begin{equation}
 I(f) = \int_{H(F) Z(\A) \bs H(\A)} \int_{H(F)Z(\A) \bs H(\A)} K(h_1, h_2) \chi(h_2) \, dh_1 \, dh_2,
\end{equation}
when convergent.  For $\gamma \in G(F)^{\mathrm{ell}}$, recall from Section \ref{local:G'} that $\gamma$ is relevant.  Choose a Haar measure on 
$H_\gamma(\A)$ and put
\[ I_\gamma(f) = \vol(Z(\A) H_\gamma(F) \bs H_\gamma(\A))\int_{H_\gamma(\A) \bs (H(\A) \times H(\A))} f(h_1^{-1} \gamma h_2) \chi( h_2) \,
dh_1 \, dh_2, \]
where $H_\gamma = \{ (h_1, h_2) \in H \times H : h_1 \gamma h_2  = \gamma \} \cong T_X$
if $\gamma = g(X)$. 
We note that $I_\gamma(f)$ factors as a finite global constant and a product of local orbital integrals,
\begin{equation} \label{orb-int-fact-eq}
I_\gamma(f)= \vol(Z(\A) H_\gamma(F) \bs H_\gamma(\A)) \prod_v I_\gamma (f_v).
\end{equation}

For $\pi$ a  cuspidal automorphic representation of $G(\A)$ with $\omega_\pi = 1$, let
\[ K_\pi(x,y) = \sum_{\phi} \pi(f) \phi(x) \overline{\phi(y)} \]
where $\phi$ runs over an orthonormal basis for the space of $\pi$, and put
\[ I_\pi(f) =  \int_{H(F) Z(\A) \bs H(\A)} \int_{H(F)Z(\A) \bs H(\A)} K_\pi(h_1, h_2) \chi(h_2) \, dh_1 \, dh_2. \]
The terms $I_\gamma(f)$ and $I_\pi(f)$
are convergent, as will be explained in the proof below.

For $v < \infty$, put $\Xi_v = 1_{\GL_{2n}(\mathcal O_v)}$.

 A  function $f_v \in C_c^\infty(G(F_v))$ is said to be a supercusp form if 
 $
 \int_{N_v} f_v(gnh)=0, 
 $
for all $g, h \in G(F_v)$
and all unipotent radicals $N_v$ of proper parabolic subgroups of $G(F_v)$.
 If $\Phi_v \in C_c^\infty(G(F_v)/Z(F_v))$ is a matrix coefficient for a supercuspidal representation $\pi_v$,
  then there exists a supercusp form $f_v$ such that $\Phi_v(g) = \int_{Z(F_v)} f_v(gz) \, dz$.  In this case,
  we say $f_v$ is essentially a matrix coefficient for $\pi_v$.

\begin{prop} \label{srtf-prop}
Let $f = \prod f_v \in C_c^\infty(G(\A))$ such that (i) $f_v = \Xi_v$ at almost all $v$;
(ii) at some  finite place $v_1$ of $F$, $f_{v_1}$ is a supercusp form; and
 (iii) at some place $v_2$ of $F$, 
$f_{v_2} \in C_c^\infty(G(F_{v_2})^{\mathrm{ell}})$.  Then
\begin{equation} \label{srtf-eq}
 \sum_{\gamma \in G(F)^{\mathrm{ell}}} I_\gamma(f) = \sum_{\pi \text{ cusp}} I_\pi(f),
\end{equation}
where $\pi$ runs  over cuspidal representations of $G(\A)$ with 
trivial central character. Here both sides are absolutely convergent.
\end{prop}

\begin{proof}
First observe for $\gamma \in G(F)^{\mathrm{ell}}$, we formally
have a factorization into local orbital integrals $I_\gamma(f) = c_\gamma \prod_v I_\gamma(f_v)$
as in \eqref{orb-int-fact-eq}.  See equations \eqref{orbsplit} and \eqref{orbnonsplit} for the definition of $I_\gamma(f_v)$. (For $v$ split, $I_\gamma(f_v)$ is defined by \eqref{orbsplit} with $\eta$ trivial. The archimedean orbital integrals are defined in the same way as the nonarchimedean ones.) For a fixed $\gamma$, at almost all $v$, $I_\gamma(f_v)$ simply reduces to either the twisted orbital
integral on $\GL_n(E_v)$ or the usual orbital integral on $\GL_n(F_v)$ for the unit element of the
Hecke algebra.  
Since each local orbital integral converges, the global integrals 
$I_\gamma(f)$ are absolutely convergent by the convergence of the elliptic terms appearing
in the trace formula in \cite{arthur:1989}.

Hence, at least formally, by (iii), $I(f)$ equals the left hand side of \eqref{srtf-eq}.
We can justify this by showing that at most finitely many $I_\gamma(f)$ are nonzero.  
It suffices to show that at most finitely many elliptic $H$-conjugacy classes of
\[ S(\A) = \{ g \theta(g)^{-1} : g \in G(\A) \} \]
lie in a given compact subset $\Omega$ of $S(\A)$.  This follows from the fact that
only finitely many elliptic conjugacy classes of $\GL_n(E)$ intersect a given compact
subgroup of $\GL_n(\A_E)$.

Lastly, it is well known that condition (ii) implies
$K(x,y) = \sum K_\pi(x,y)$ where $\pi$ runs over cuspidal representations,
that
and each $K_\pi(x,y)$ is rapidly decreasing.  This makes $I(f) = \sum_\pi I_\pi(f)$, where
the sum is absolutely convergent.
\end{proof}

We use this result to get the existence of many $H(F_v)$-elliptic representations.  
 
\begin{prop} \label{ell-supp-cor} Suppose $k$ is a local field of characteristic $0$, let $K/k$
be a quadratic \'etale extension, and let $D(k)$ be the split or non-split quaternion algebra over $k$, which we take to be split if $K/k$ is.
 There exist irreducible admissible unitary representations $\tau$ of $\GL_n(D(k))$ which are 
$\GL_n(K)$-elliptic.  
\end{prop}

\begin{proof} We may globalize $k$, $K$ and $D(k)$ to $F$, $E$ and $D$
such that (a) these local algebras are the localizations of the corresponding global algebras at
some place $v_1$, (b) there is another place 
$v_2$ of $F$ over which $D$ splits and such that there exists an $H(F_{v_2})$-elliptic 
supercuspidal representation
$\tau_2$ of $G(F_{v_2})$ (Proposition \ref{ell-supp-prop}), and (c) there is some infinite place $v_3 \ne v_1$ such that 
$D_{v_3}$ is split. 

Choose a test function $f = \prod f_v$ as follows.  Let $f_{v_1}$ be the characteristic function
of an open compact subset of $G(F_{v_1})^{\ell}$.  By Lemma \ref{ell-int-lem}   we can take $f_{v_2}$ to be essentially a matrix coefficient of $\tau_2$
such that $I_\gamma(f_{v_2})$ is nonzero for any $\gamma \in \Omega_{v_2}$,
where $\Omega_{v_2}$ is some open subset of $G(F_{v_2})^{\ell}$.  At all other finite $v$,
choose $f_v$ to be a characteristic function of some compact subset of $G(F_v)$ such that
$f_v = \Xi_v$ outside of some finite set of places $S$.  The archimedean choices will be made
below.

Let $C \subset G(\A^\infty)$ be the support of $f^\infty = \prod_{v < \infty} f_v$.  Note 
$Z(\A^\infty)C \cap \SL_n(D(\A^\infty))$ is open in $\SL_n(D(\A^\infty))$.  
Strong approximation for $\SL_{n}(D)$ for indefinite $D$ 
tells us $\SL_n(D(F))$ is dense in $\SL_n(D(\A^\infty))$,
so there  exists $\gamma \in G(F) \cap Z(\A^\infty) C \subset G(F)^\ell$. 

Thus, for any such $\gamma$, 
we must have $I_\gamma(f_v) \ne 0$ at any $v$ where $f_v$ is a characteristic function.  
The only other finite place to consider is $v_2$, but we can guarantee
$I_{\gamma}(f_{v_2}) \ne 0$ by taking $\gamma \in \Omega_{v_2}$.  Fix one such $\gamma$.
For $v | \infty$, choose $f_v$ such that $I_{\gamma}(f_v) \ne 0$ and $I_{\gamma_1}(f) = 0$ 
for any $\gamma_1 \in G(F) - H(F) \gamma H(F)$.  
This is possible by taking the support of archimedean $f_{v_3}$
small enough, since only a finite number of global geometric terms $I_{\gamma_1}(f)$ can be nonzero
as explained in the proof of the previous proposition.  

At almost all places, we have $I_{\gamma}(f_v) \ge \vol ((H_\gamma \bs H \times H)(\O_{v}))$,
and the product over $v$ is nonzero.
Hence, for $f$ and $\gamma$ as above,  $I(f) = I_\gamma(f) \ne 0$.
By \eqref{srtf-eq}, $I_\pi(f) \ne 0$
for some cuspidal $\pi$.  By the uniqueness of local $H(F_v)$-invariant functionals,
$I_\pi(f)$ factors  into a product of local Bessel distributions $B_{\pi_v}(f_v)$, and thus
$B_{\pi_{v_1}}(f_{v_1}) \ne 0$.
\end{proof}

%
%

\section{Main results} \label{main-sec}

%
%

Let $G$, $H$, $G'$ and $H'$ be as in the introduction and choose measures as
in the previous section.  Assume $E$ is a field which is split at
each archimedean place.  On $G$, we keep the same notation for the (partial) distributions
$I_*$ defined in the previous section; on $G'$, we denote them with primes, 
i.e., by $I'_*$.  Write $\Sigma_s$ for the set of places of $F$ split in $E$ and $\Sigma_s^c$ for the set of places of $F$ inert or ramified in $E$.

Recall all representations are assumed to be unitary with trivial central character.

\subsection{Global results}

\begin{thm} \label{main-thm1}  Fix $D \in \Xef$ and say $G=G_D$.
Suppose $\pi$ is an $H$-distinguished cuspidal automorphic representation of $G(\A)$, 
which is supercuspidal at some finite place $v_1$ where $E/F$ is split
and $H(F_{v_2})$-elliptic at another place $v_2$. 
Let $\pi' = \JL(\pi)$ be the Jacquet--Langlands transfer to $G'(\A)$.  
Then 
$\pi'$ is $H'$-distinguished and $(H',\eta)$-distinguished.
\end{thm}

\begin{proof}  Since $\pi$ is $H$-distinguished, $I_\pi \not \equiv 0$.
We will choose a nice test function $f= \prod f_v \in C_c^\infty(G(\A))$ such that $I_\pi(f) \ne 0$.
By the uniqueness of $H$-invariant linear forms on $\pi$, we can factor $I_\pi(f) = \prod B_{\pi_v}(f_v)$
where the $B_{\pi_v}$'s are local Bessel distributions attached to certain linear functionals 
$\lambda_v = \lambda_{1,v} = \lambda_{2,v}$ as in Section \ref{bess-sec}.  

At $v_1$, we may take $f_{v_1}$ to be essentially a matrix coefficient 
(in the sense of Section \ref{srtf-sec}) of $\pi_{v_1}$ such that $B_{\pi_{v_1}}(f_{v_1}) \ne 0$. 
At $v_2$, we may take $f_{v_2} \in C_c^\infty(G(F_{v_2})^{\mathrm{ell}})$ such that
$B_{\pi_{v_2}}(f_{v_2}) \ne 0$.

There exists a finite set of places $S$, including all archimedean places,  
such that for $v \not \in S$, $B_{\pi_v}(\Xi_v) \ne 0$.  Enlarge $S$ if necessary so
that it contains all even places and all places where $E$ or $D$ ramifies.
For $v \not \in S$, take $f_v = \Xi_v$.
Away from $S$, choose $\lambda_v$ so that
$B_{\pi_v}(\Xi_v) = 1$ to ensure convergence of the factorization of $I_\pi(f)$.  Now 
consider $v \in S - \{ v_1, v_2 \}$.  If $v \in \Sigma_s$,  
take any $f_v$ such that $B_{\pi_v}(f_v) \ne 0$.  If $v \not \in \Sigma_s$, 
we may choose
$f_v \in C_c^\infty(G(F_v)^{\mathrm{main}})$ such that $B_{\pi_v}(f_v) \ne 0$ by Proposition \ref{Bpi-supp-prop}.
This defines $f$ such that $I_\pi(f) \ne 0$.

Now we will get a matching $f'$. For regular $\gamma$ and $\gamma'$ whose
double cosets correspond at $v$, choose
measures on $H_\gamma(F_v)$ and $H'_{\gamma'}(F_v)$ which are compatible.
Whenever $E/F$ is split, we can identify $G(F_v)$ and $G'(F_v)$ so that $H(F_v) = H'(F_v)$.  At such places,
take $f'_v = f_v$.  When $v \not \in S$ is inert, the function $f'_v = \Xi_v$ matches $f_v=\Xi_v$ by 
 Guo's fundamental lemma (Proposition \ref{prop:fund}, again identifying $G(F_v)$ with
 $G'(F_v)$).
When $v \in S$ is inert (and thus nonarchimedean by assumption), we know there exists a matching
$f'_v$ for $f_v$ by Proposition \ref{prop:matchGtoG'} when $v$ is odd; for any nonarchimedean $v$ this follows from \cite{chong-zhang}.
We may also assume $f'_{v_2} \in C_c^\infty(G'(F_{v_2})^{\mathrm{ell}})$.  
Let $f' = \prod f'_v \in C_c^\infty(G'(\A))$.   By the equality of the global volumes of stabilizers for 
matching elliptic elements, we have that $f'$ matches $f$ globally, in the sense that
$I_\gamma(f) = I'_{\gamma'}(f')$ for each regular matching $\gamma$ and $\gamma'$.

Therefore, Proposition \ref{srtf-prop} implies 
\begin{equation}\label{eqn:speceq}
 \sum_{\sigma \text{ cusp}} I_\sigma(f) = \sum_{\sigma' \text{ cusp}} I'_{\sigma'}(f').
\end{equation}
Let $S^c$ denote the complement of $S$. For $v \in \Sigma_s \cap S^c$, we may vary $f_v$ in the Hecke algebra and retain \eqref{eqn:speceq}.  Therefore,
the principle of linear independence of characters (see \cite[Lemma 4]{lapid:rogawski}) implies
\begin{equation} \label{thm61b}
 \sum_{\sigma \in \Pi} I_\sigma(f) = \sum_{\sigma' \in \Pi'} I'_{\sigma'}(f')
\end{equation}
where $\Pi$ (resp.\ $\Pi'$) denotes the set of cuspidal representations of $G(\A)$ (resp.\ $G'(\A)$) which,
at each $v \in \Sigma_s \cap S^c$, are isomorphic to $\pi_v$.  A result of Ramakrishnan for $\GL(n)$ \cite{ramakrishnan:mult1} tells us that if $\tau_1, \tau_2 \in \Pi'$, then
$\tau_2 \cong \tau_1$ or $\tau_2 \cong \tau_1 \otimes \eta$.
Hence by strong multiplicity one for $G'$, we have $\Pi' = \{ \pi', \pi' \otimes \eta \}$. 
Using strong multliplicity one for $G$ and the Jacquet--Langlands transfer to $G'$, we can also apply
\cite{ramakrishnan:mult1} to get that $\Pi = \{ \pi, \pi \otimes \eta \}$.
Thus \eqref{thm61b} becomes
\begin{equation}
 I_\pi(f) + I_{\pi \otimes \eta}(f) = I'_{\pi'}(f') + I'_{\pi' \otimes \eta}(f').
\end{equation}
Since $I'_{\pi'} \not \equiv 0$ if and only if $I'_{\pi' \otimes \eta} \not \equiv 0$
if and only if $\pi'$ is $H'$- and $(H',\eta)$-distinguished, 
we want to show $I_\pi(f) + I_{\pi \otimes \eta}(f) \ne 0$.  If $\pi \cong \pi \otimes \eta$, then
$I_{\pi \otimes \eta}(f) = I_\pi(f) \ne 0$.  If not, it is a priori possible that
$I_{\pi \otimes \eta}(f) = - I_\pi(f)$.  However, by Lemma \ref{Bpi-pm-lem}, we may choose 
$f_{v_3} \in C_c^\infty(G(F_{v_3})^{\mathrm{main}})$ at some odd non-split place $v_3$ to ensure that $I_{\pi \otimes \eta}(f) \ne - I_\pi(f)$.
\end{proof}

\begin{rem} While in light of \cite{chong-zhang} we do not need to appeal to
Proposition \ref{prop:matchGtoG'} for matching at odd places, we use it in the argument
because (i) \cite{chong-zhang} was not available at the time of the first version of our paper,
and (ii) we hope this approach may provide a simpler way to get global results in situations
where smooth matching is not known.  Without using \cite{chong-zhang}, the above argument still goes through with the additional hypotheses that at each
even place $v$ either $v$ is split or $\pi_v$ is $H(F_v)$-elliptic.
\end{rem}

\begin{prop} \label{main-thm2}
Let $\pi'$ be a cuspidal automorphic representation of $G'(\A)$ with trivial central character such that $\pi'$ is both $H'$- and $(H',\eta)$-distinguished, 
$\pi'$ is supercuspidal at a split place $v_1$, and $\pi'$ is $H(F_{v_2})$-elliptic at another
split place $v_2$.  
Assume that, for each $v$ inert in $E$, at least one of the following holds:

\medskip

(a) $v$ is an odd place at which $E/F$ is unramified and $B_{\pi'_v}(\Xi_v) \ne 0$;

(b) $B_{\pi'_v}$ is not identically zero on $C_c^\infty(X_v)$ 
where 
\[ X_v = \left\{ g \in G'(F_v)^{\mathrm{main}} : [g]  \in \bigcup_{\eps \in F_v^\times /
N E_v^\times} \iota_\eps  (\Gamma^{\mathrm{ss}}(G_\eps(F_v))) \right\}. \]

\medskip
\noindent
Then, there exists $D \in \Xefpi$
such that the Jacquet--Langlands transfer $\pi_D$ to $G_D(\A)$  is 
$H$-distinguished.
\end{prop}

Note (a) holds for almost all $v$ and (b) is automatically 
satisfied when $n$ is odd and $\pi'_v$ is $H'_v$-elliptic. 
In particular, this establishes Conjecture \ref{guo-conj}(2) under some local hypotheses
on $\pi'$ (admittedly, stronger hypotheses than one would like).

\begin{proof}
The proof is similar to the previous case.  We just explain where details differ.

Factor $I_{\pi'}(f')$ into local Bessel distributions $B_{\pi'_v}(f'_v)$ as defined in Section
\ref{bess-sec}.
We choose $f'_v$ such that $B_{\pi'_v}(f'_v) \ne 0$ at each place $v$ with the conditions that
(i)  $f'_v = \Xi_v$ outside of some finite set of places $S$ containing $v_1$ and $v_2$
(assume $S$ is small enough so every $v \in S \cap \Sigma_s^c$ satisfies (b));
(ii) $f'_{v_1}$ is a supercusp form;  (iii) $f'_{v_2} \in C_c^\infty(G'(F_{v_2})^{\mathrm{ell}})$;
and (iv) at any $v \in S \cap \Sigma_s^c$ satisfying (b), 
$f'_v \in C_c^\infty(X_v)$.

By Propositions \ref{prop:fund} and \ref{prop:matchG'toG} or \ref{prop:matchG'toGeven} 
(see also the partial converse in \cite{chong-zhang}), for each $v$ there is a pair of matching functions $(f_{v, \eps_1}, f_{v, \eps_2})$ that satisfy the following conditions: (i) $(f_{v_1, \eps_1}, f_{v_1, \eps_2})=(f_{v_1}',0)$; and (ii) $f_{v_2, \eps_1}$ has elliptic support. 

For $D_\eps\in \Xef$ we let $f_{D_\eps}=\prod_v f_{v, \eps}$. Denote
the distributions $I$ and $I_{\pi_D}$ defined in the previous section on $G_D$ by $I_{D}$ and $I_{D, \pi_D}$ respectively. Thus we have
\begin{equation}
 \sum_D I_D(f_D) = I'(f'),
\end{equation}
where in fact the sum on the left is finite.  
Again, by \cite{ramakrishnan:mult1} and strong multiplicity one, one gets
\begin{equation}
 \sum_{D \in X(E:F:\pi')} I_{D,\pi_D}(f_D) + I_{D,\pi_D \otimes \eta}(f_D) =  I'_{\pi'}(f') + I'_{\pi' \otimes \eta}(f').
\end{equation}
By Lemma \ref{Bpi-pm-lem}, the right hand side may be chosen nonzero, so that for at least one
such $D$, we have $I_{D,\pi_D} \not \equiv 0$ or $I_{D,\pi_D \otimes \eta} \not \equiv 0$.  But these conditions
are both equivalent to $\pi_D$ being $H$-distinguished.
\end{proof}

The following result tells us that when an analogue of Conjecture \ref{guo-conj}(2) for $n$ even holds, the $D$ should not be unique.

\begin{thm} \label{main-thm3}
Suppose $n$ is even and $D_1, D_2 \in \Xef$.  Let $\pi_{D_1}$ and $\pi_{D_2}$
be cuspidal automorphic representations of $G_{D_1}(\A)$ and $G_{D_2}(\A)$  
such that $\JL(\pi_{D_1}) \cong \JL(\pi_{D_2})$.
Suppose $\pi_{D_1}$ is  (i) supercuspidal at a place $v_1$ 
 such that $D_{1}(F_{v_1})\cong D_{2}(F_{v_1})$ and  (ii)  $H(F_v)$-elliptic for $v$ such that $D_1(F_v)\not\cong D_2(F_v)$. Then if $\pi_{D_1}$ is $H$-distinguished, so is $\pi_{D_2}$.\end{thm}
\begin{proof}
Assume $D_1 \not\cong D_2$. Again the proof is similar to the previous cases and we just explain where details differ.  Here we directly compare the trace formulas on $G_{D_1}$ and $G_{D_2}$. We
construct matching $f_{1,v}\in C_c^\infty(G_{D_1}(F_v))$ and $f_{2,v} \in C_c^\infty(G_{D_2}(F_v))$ such that $I_{g_{\eps_1}(X)}(f_{1,v})=I_{g_{\eps_2}(X)}(f_{2,v})$ for all $X \in \Gamma^{\ell, \twist}(\GL_n(E))$, and the global orbital integrals vanish on non-elliptic
terms. Let $S=\{ v  :D_{1}(F_v)\not  \cong D_{2}(F_v)\}$. 
We choose $f_1 \in C_c^\infty(G_{D_1}(\A))$ such that (i) for all $v$, $B_{\pi_{D_1,v}}(f_{1,v})\neq 0$;  (ii)  for almost all $v \not \in S$, $f_{1,v}=\Xi_v$; (iii) $f_{1,v_1}$ is a supercusp form;  and (iv)
  $f_{1,v} \in C_c^\infty(G_{D_1}(F_{v})^{\ell})$  for $v \in S$.
   There is a matching $f_2$ by taking
$f_{2,v}=f_{1,v}$ for all $v \in S^c$ 
and using Corollary \ref{cor:match} for the 
remaining $v$.
\end{proof}

\begin{rem} \label{final-rem}
Note Theorem \ref{main-thm3} remains valid if the only archimedean assumption
one makes is $D_{1}(F_v) \cong D_{2}(F_v)$ for each $v | \infty$, i.e., we need not
assume $E/F$ is split at each infinite place.
\end{rem}

\subsection{Local results}
\label{loc-res-sec}

Here we deduce some local consequences of our global results. 

Let $K/k$ be a quadratic extension of nonarchimedean local fields of characteristic 0, and $\eta_{K/k}$ the associated quadratic
character of $k^\times$.  Then we may choose our quadratic extension of number fields $E/F$ such that, for a fixed
place $v_0$ of $F$, one has $F_{v_0} \cong k$, $E_{v_0} \cong K$, and $E/F$ is split
at each archimedean place and each even place except possibly $v_0$.  We will
also fix an odd split place $v_1$ and
assume $F$ has a split even place $v_2\neq v_0$ 
such that $F_{v_2} \cong \Q_2$.  Identify $k=F_{v_0}$ and $K = E_{v_0}$.

Take $D \in \Xef$, and $G, G', H, H'$ as before.  Let $\tau$ and $\tau'$
 be irreducible admissible representations of $G(k)$  and $G'(k)$.
Recall $\tau$ is $H(k)$-distinguished if $\Hom_{H(k)}(\tau, \C)\neq 0$.  Similarly, $\tau'$ is $H'(k)$- (resp.\ $(H'(k),\eta_{K/k})$-) distinguished if 
$\Hom_{H'(k)}(\tau', \C)$ (resp.\ $\Hom_{H'(k)}(\tau',\eta_{K/k})$) is nonzero.

\begin{thm} \label{local-thm1}
Let $\tau$ be a supercuspidal representation of $G(k)$,
and $\tau'$ be its Jacquet--Langlands transfer to $G'(k)$.
If $\tau$ is $H(k)$ distinguished, then $\tau'$ is both
 $H'(k)$- and $(H'(k), \eta_{K/k})$-distinguished.
\end{thm}

\begin{proof}
Let $\pi_{v_0} = \tau$.  By \cite{hakim:2002b}, there exists an $H(F_{v_1})$-distinguished (tame) supercuspidal representation $\pi_{v_1}$
of $G(F_{v_1})$. 
(Murnaghan  pointed out to us that one can also deduce this fact from \cite{murnaghan:2011}.)
By Proposition \ref{ell-supp-prop}, there also exists an 
$H(F_{v_2})$-elliptic (simple) supercuspidal representation $\pi_{v_2}$
of $G(F_{v_2})$.  
An argument of 
Hakim and Murnaghan \cite{hakim:2002} (see \cite[Theorem 4.1]{PSP} for a more
general form)
shows that $\pi_{v_0}$, $\pi_{v_1}$ and $\pi_{v_2}$ can be simultaneously
globalized to an $H$-distinguished representation
$\pi$ of $G(\A)$.
Then, by Theorem  \ref{main-thm1}, $\pi' = \JL(\pi)$ is $H'$-
and $(H',\eta)$-distinguished.  In particular $\pi'_{v_0}\cong \tau'$ is locally $H'(F_{v_0})$- and $(H'(F_{v_0}),\eta_{v_0})$-distinguished.
\end{proof}

\begin{thm} \label{local-thm2}
Let $D_1(k)$ and $D_2(k)$ be the two quaternion algebras over $k$, in some order.
Suppose $n$ is even, and $\tau_{D_1}$ and $\tau_{D_2}$ are irreducible admissible representations of $G_{D_1}(k)$ and $G_{D_2}(k)$
which correspond via Jacquet-Langlands. Assume $\tau_{D_1}$ is $H(k)$-elliptic and
supercuspidal.  Then if $\tau_{D_1}$ is $H(k)$-distinguished,
so is $\tau_{D_2}$.
\end{thm}

\begin{proof}
The proof is similar to the previous proof, with the following modifications:
We globalize $k$ to a number field $F$ as above such that $F_{w} \cong F_{v_0} \cong k$
for some place $w \ne v_0$.  Globalize $D_1$ and $D_2$ so that one is split
everywhere and one is ramified only at $v_0$ and $w$.  Globalize $\tau_{D_1}$ to
$\pi_{D_1}$ such that $\pi_{D_1,v_0} \cong \pi_{D_1,w} \cong \tau_{D_1}$.  
Then argue as above, applying Theorem \ref{main-thm3} at the end.
\end{proof}

Now we prove our final local result. 

\begin{proof}[Proof of Theorem \ref{intro-local-thm2}]
The globalization result  used above (\cite[Theorem 4.1]{PSP})
in fact tells us that there exists a cuspidal, globally $\GL_n(\A_E)$-distinguished representation $\pi$ of $\GL_{2n}(\A)$ such that, $\pi_{v_0} \cong \tau$,
 $\pi_{v_1}$ is an $H(F_{v_1})$-distinguished supercuspidal,
 $\pi_{v_2}$ is an $H(F_{v_2})$-elliptic supercuspidal, and 
 $\pi_v$ is unramified for  all finite $v \not \in \{ v_0, v_1, v_2 \}$.  
 By Theorem \ref{main-thm1}, we know $\pi$ must also be
 $H'$- and $(H', \eta)$-distinguished.  Hence $\pi$ is symplectic and $L(1/2, \pi_E) \ne 0$.
In particular, $\pi$ has a nonzero global Shalika period by \cite{jacquet-shalika}.
 Therefore $\tau$ has a nonzero local Shalika period, which means $\tau$ is symplectic by
 \cite{JNQ}.
 
Now $L(1/2, \pi_E) \ne 0$ implies the global root number $\epsilon(1/2, \pi_E) = +1$.  The global root number
 factors into a product of local root numbers $\epsilon(1/2, \pi_{E,v})$ (independent of
 choice of local additive character by self-duality).  Thus to show 
 $\epsilon(1/2, \pi_{E,v_0}) = \epsilon(1/2, \tau_K) = 1$, it suffices to show 
 $\epsilon(1/2, \pi_{E,v}) = +1$ for all $v \ne v_0$.
 
 Because each $\pi_v$ and $\pi_{E,v}$ is self-dual, we have that each
 $\epsilon(1/2, \pi_v)$ and $\epsilon(1/2, \pi_{E,v})$ is $\pm 1$.
Since we are working in even dimension,
$\epsilon(1/2, \pi_{E,v}) = \epsilon(1/2, \pi_v) \epsilon(1/2, \pi_v \otimes \eta_v)$.

If $E/F$ is split at $v$, then $\epsilon(1/2, \pi_{E,v}) = \epsilon(1/2, \pi_v)^2 = +1$.

If $E/F$ is inert at $v \ne v_0$, then $\pi_v$ and $\pi_v \otimes \eta_v$ are both 
unramified and thus have local root number $+1$,
whence $\epsilon(1/2, \pi_{E,v}) = +1$.

Last, suppose $E/F$ is ramified at $v$ and $v \ne v_0$.  Then 
$\pi_v \cong \pi(\chi_1, \ldots, \chi_{2n})$ is an unramified principal series, so 
$\epsilon(1/2, \pi_v)=1$.  It follows that
\[ \epsilon(1/2, \pi_v \otimes \eta_v) = \prod \epsilon(1/2, \chi_i \eta_v)
= \epsilon(1/2, \eta_v)^{2n} \prod \chi_i(\varpi_v), \]
where $\varpi_v$ is a uniformizer.  But this is $+1$ since $\epsilon(1/2, \eta_v) = \pm 1$
and $\prod \chi_i = 1$ as $\pi$ has trivial central character.
\end{proof}

\providecommand{\bysame}{\leavevmode\hbox to3em{\hrulefill}\thinspace}
\providecommand{\MR}{\relax\ifhmode\unskip\space\fi MR }
\providecommand{\MRhref}[2]{%
  \href{http://www.ams.org/mathscinet-getitem?mr=#1}{#2}
}
\providecommand{\href}[2]{#2}

\end{document}